\documentclass{amsart}

\usepackage{amsfonts, amssymb, amsmath, eucal, verbatim, amsthm, amscd, enumerate, mathtools}

\usepackage{graphicx}
\usepackage{framed}
\usepackage{tikz}
\usepackage{enumitem}
\usepackage{bbm}

\usepackage{ascmac}

\newtheorem{theorem}{Theorem}[section]
\newtheorem{lemma}[theorem]{Lemma}
\newtheorem{proposition}[theorem]{Proposition}

\newtheorem{corollary}[theorem]{Corollary}
\theoremstyle{definition}

\theoremstyle{remark}
\newtheorem*{remark}{Remark}

\newcommand{\vertiii}[1]{{\left\vert\kern-0.25ex\left\vert\kern-0.25ex\left\vert #1 
\right\vert\kern-0.25ex\right\vert\kern-0.25ex\right\vert}}
\newcommand{\R}{{\mathbb R}}

\numberwithin{equation}{section}

\usepackage{setspace}

\def\1{\textbf{\rm 1}}

\def\XXint#1#2#3{{\setbox0=\hbox{$#1{#2#3}{\int}$}
\vcenter{\hbox{$#2#3$}}\kern-.5\wd0}}

\makeatletter
\@namedef{subjclassname@2020}{\textup{2020} Mathematics Subject Classification}
\makeatother

\begin{document}



\author[Nakamura]{Shohei Nakamura}
\address[Shohei Nakamura]{Department of Mathematics, Graduate School of Science, Osaka University, Toyonaka, Osaka 560-0043, Japan}
\email{srmkn@math.sci.osaka-u.ac.jp}
\author[Tsuji]{Hiroshi Tsuji}
\address[Hiroshi Tsuji]{Department of Mathematics, Graduate School of Science and Engineering, Saitama University, Saitama 338-8570, Japan}
\email{tsujihiroshi@mail.saitama-u.ac.jp}

\title[Generalised Legendre duality and Gaussian saturation]{A generalized Legendre duality relation and Gaussian saturation 
}

\begin{abstract}
%
Motivated by the barycenter problem in optimal transportation theory, Kolesnikov--Werner recently extended the notion of the Legendre duality relation for two functions to the case for multiple functions. 
We further generalize the duality relation and then establish the centered Gaussian saturation property for a Blaschke--Santal\'{o} type inequality associated with it. 
Our approach to the understanding such a generalized Legendre duality relation is based on our earlier observation that directly links Legendre duality with the inverse Brascamp--Lieb inequality. 
More precisely, for a large family of degenerate Brascamp--Lieb data, we prove that the centered Gaussian saturation property for the inverse Brascamp--Lieb inequality holds true when inputs are restricted to even and log-concave functions. 


As an application to convex geometry, we establish the most important case of a conjecture of Kolesnikov and Werner about the Blaschke--Santal\'{o} inequality for multiple even functions as well as multiple symmetric convex bodies. 
Furthermore, in the direction of information theory and optimal transportation theory, this provides an affirmative answer to another conjecture of Kolesnikov--Werner about a Talagrand type inequality for multiple even probability measures that involves the Wasserstein barycenter. 
\end{abstract}

\maketitle

\section{Introduction}
Duality is a pervasive concept in mathematics, appearing in many different forms across various fields. 
In convex geometry, it is represented by the notion of the polar body that is associated to any convex body in Euclidean spaces.
A celebrated Blaschke--Santal\'{o} inequality describes the correlation between a convex body and its dual in terms of their volumes. It states that the product of volumes of the convex body and its polar body, which is so-called the volume product or Mahler volume, is maximized by Euclidean ball among all symmetric convex bodies. 
It is K. Ball \cite{BallPhd} and Artstein-Avidan--Klartag--Milman \cite{AKM} who extended this duality inequality for convex bodies to the functional analytic framework, where the dual of a function is described by the Legendre transform. 
As an analogue to the volume product of a convex body, they introduced the functional volume product of a function via the Legendre duality relation, and established the functional Blaschke--Santal\'{o} inequality; the functional volume product is maximized by the centered Gaussian among all even functions. 
Among several applications of the functional Blaschke--Santal\'{o} inequality, we mention the application / interpretation in information theory and geometry of probability measures via  entropy and Wasserstein distance, that is so-called the symmetric Talagrand inequality \cite{Fathi}. 
Such an interpretation is underpinned by the Kantorovich duality where the Legendre duality relation plays an important role to describe the Wasserstein distance.  
Furthermore, the symmetric Talagrand inequality may be also regarded as the effective inequality to understand midpoints of two probability measures on the Wasserstein space. 
Such an observation was made by Kolesnikov--Werner \cite{KW}, and they pursued further developments toward this direction by investigating the problem about the Wasserstein barycenter of multiple probability measures. 
According to the purely mathematical interest, as well as the recent realization of its usefulness 
in theoretical computer science \cite{Strum, PeyCut}, the study of the barycenter problem has attracted attention in optimal transportation theory. 
Through the Kantorovich duality and this barycenter problem, Kolesnikov--Werner \cite{KW} introduced a notion of the generalized Legendre duality relation for multiple functions, and proposed the functional Blaschke--Santal\'{o} type inequality associated with it. 
After observing the equivalence to their Talagrand type inequality for multiple even probability measures that involves the barycenter, they established their generalized Blaschke--Santal\'{o} type inequality when the input functions are all unconditional. 
It has been conjectured in \cite{KW} that the inequality holds true for all even input functions, and this remains to be still open as far as we are aware. 

In this paper, we develop our new observation made in \cite{NT,NT2} about the direct link\footnote{To be fair, it is well-known that the Pr\'{e}kopa--Leindler inequality, which is a member of so-called Barthe's reverse Brascamp--Lieb inequality, is useful to the study of the Legendre transform. As will be clear, our direct link is not about this, and related to so-called the inverse Brascamp--Lieb inequality.} between the classical Legendre duality relation and the Brascamp--Lieb theory regarding the multilinear integral functional, and then advance the study of the generalized Legendre duality relation. 
This new link enables us to employ deep ideas and techniques that have been developed in the Brascamp--Lieb theory for the purpose of the study of the generalized Legendre duality relation. 
In particular, we use so-called Ball's inequality\footnote{
This is a certain monotonicity statement of the Brascamp--Lieb constant under the convolution. 
}, whose strength has been recently capitalized in the study of the nonlinear Brascamp--Lieb inequality by Bennett et al. \cite{BBBCF}, as a fundamental idea.  
In more precise terms, we have two main results in this paper. We first extend a  generalized Legendre duality relation of Kolesnikov--Werner to much wider class of duality relations based on the spirit of the Brascamp--Lieb theory. Then  our first main result states that the best constant of the Blaschke--Santal\'{o} type inequality associated with this generalized Legendre duality relation is saturated by centered Gaussians. 
This property is universal in the sense that it holds independently of duality relations, and the idea of such a property clearly comes from Lieb's fundamental theorem \cite{Lieb}. 
This first main result is a consequence of our second main result about the centered Gaussian saturation property for the inverse Brascamp--Lieb inequality under the evenness. 
The study of the inverse Brascamp--Lieb inequality has been initiated by Chen--Dafnis--Paouris \cite{CDP}, and then Barthe--Wolff have done the systematic study of the inequality by imposing some nondegeneracy condition  \cite{BW}. 
The crucial point here is that, although their nondegeneracy condition is necessary for their study in a certain sense, the direct link between the Legendre duality relation and the inverse Brascamp--Lieb inequality appears at the degenerate case, where the systematic study of Barthe--Wolff is not applicable. 
Nevertheless, for a large class of Brascamp--Lieb data that do not necessarily satisfy the nondegeneracy condition of Barthe--Wolff, we manage to establish the centered Gaussian saturation property for the inverse Brascamp--Lieb inequality by imposing the evenness and log-concavity on inputs.  
As an application in convex geometry, we establish the special, but the most important, case of a conjecture of Kolesnikov--Werner about the Blaschke--Santal\'{o} inequality for multiple even functions and multiple symmetric convex bodies. In the direction of information theory and optimal transportation theory, this settles down their conjecture about a Talagrand type inequality for multiple even probability measures that involves the barycenter. 

\subsection{Generalized Legendre duality relation and Blaschke--Santal\'{o} type inequality}\label{Section1.1}
Duality is an ubiquitous phenomenon that appears in many area of mathematics. 
The fundamental idea of the use of duality is to extract deeper information of some mathematical object itself by investigating its dual object. 
Therefore it is pivotal to understand the relation between the original mathematical object and its dual object. 
In convex geometry, such duality is described by the notion of the polar body. 
For a given symmetric convex body $K\subset \mathbb{R}^n$, where $K$ is said to be symmetric if $-K=K$, its polar body is defined as $K^\circ := \{ x\in \mathbb{R}^n: \sup_{y\in K}\langle x, y\rangle \le 1 \}$, where $\langle \cdot,\cdot \rangle$ is a standard inner product of $\mathbb{R}^n$.
One way to understand the relation between $K$ and $K^\circ$ is to consider the quantity $v(K):= |K||K^\circ|$ which is so-called volume product or Mahler volume, where $|\cdot|$ stands for the standard Euclidean volume. 
Since the volume product is linear invariant, it makes sense to ask what is the maximum / minimum value of $v(K)$ among all symmetric convex bodies. 
The celebrated Blaschke--Santal\'{o} inequality provides the answer to the maximum, and states that $v(K)\le v(\mathbf{B}^n_2)$ holds for any symmetric convex body $K$ where $\mathbf{B}^n_2:=\{ x\in \mathbb{R}^n: \sum_{i=1}^n |x_i|^2 \le 1 \}$. 
This inequality was proved by Blaschke \cite{Blaschke} for $n=2,3$ and Santal\'{o} \cite{Santalo} for $n\ge4$. We refer to \cite{BK,CKLR,CFL,CGNT,CFM,MeyPaBook,MeyPa,NT2,Saint} for several alternative proofs. 
The problem to identify the minimum value of the volume product among symmetric convex bodies, known as Mahler's conjecture, is still an open problem, and has been for almost a century. Mahler expected that the minimum is attained by the Euclidean cube and confirmed it when $n=2$ \cite{Mar1,Mar2}. 
A recent breakthrough was brought by Iriyeh--Shibata \cite{IriShi} where Mahler's conjecture was confirmed affirmatively when $n=3$, and their proof was significantly simplified by Fradelizi et al. \cite{FHMRZ}. The problem for $n\ge4$ is open despite several partial progresses; see the survey article \cite{FMZSurvey}. 

The Blaschke--Santal\'{o} inequality has been put into the analytically functional framework by Ball \cite{BallPhd} and Artstein-Avidan--Klartag--Milman \cite{AKM}, see also Fradelizi--Meyer \cite{FraMeyMathZ} and Lehec \cite{LehecDirect,LehecYaoYao} for further generalizations as well as alternative proofs. 
For a nonnegative function $f$ on $\mathbb{R}^n$, its polar function, denoted by $f^\circ$, is defined as 
$$
f^\circ(x) := \inf_{y\in\mathbb{R}^n} \frac{e^{-\langle x,y\rangle}}{f(y)},\;\;\; x\in \mathbb{R}^n. 
$$
We often identify $f = e^{-\varphi}$ for some $\varphi: \mathbb{R}^n\to \mathbb{R}\cup\{+\infty\}$ and say that $f$ is log-concave if $\varphi$ is convex on $\{ \varphi < +\infty \}$. 
In this terminology, $f^\circ(x) = e^{-\varphi^*(x)}$ holds where $\varphi^*(x):= \sup_{y\in \mathbb{R}^n} [\langle x,y\rangle - \varphi(y)]$ is the Legendre transform of $\varphi$.
The functional volume product for $f$ is defined as 
$$
v(f):= \int_{\mathbb{R}^n} f\, dx \int_{\mathbb{R}^n} f^\circ\, dx. 
$$
For a symmetric convex body $K \subset \mathbb{R}^n$, the Minkowski functional $\|x\|_K:= \inf\{r>0: x\in rK\}$, $x\in\mathbb{R}^n$,  becomes a norm on $\mathbb{R}^n$ and satisfies 
\begin{equation}\label{e:Func->Geo}
\int_{\mathbb{R}^n} e^{-\frac12 \|x\|_K^2}\, dx = \frac{(2\pi)^\frac{n}2}{|\mathbf{B}^n_2|}|K|,
\;\;\;
\big( \frac12 \|\cdot\|_K^2 \big)^*(x)
= 
\frac12 \|x\|_{K^\circ}^2.
\end{equation}
It is clear from these properties that the standard Gaussian $g(x):=  e^{ -\frac12 |x|^2 }$ plays the role of $\mathbf{B}^n_2$ in this functional formulation. 
More generally, for a positive definite matrix\footnote{In this paper we write $A>0$ and $A\ge0$ if $A$ is a symmetric positive definite and semidefinite respectively. } $A$, we denote the centered Gaussian with the covariance matrix $A^{-1}$ by $g_A(x):= e^{ -\frac12 \langle x,Ax\rangle }$. 
Then the functional Blaschke--Santal\'{o} inequality states that 
 \begin{equation}\label{e:FBS}
    v(f)\le v(e^{-\frac12|x|^2}) = (2\pi)^n,
 \end{equation}
 holds for all nonnegative and even $f \in L^1(\mathbb{R}^n)$ with $\int_{\mathbb{R}^n} f\, dx >0$. 
The case of equality in \eqref{e:FBS} appears if and only if $f$ is  multiplicative of $g_A$ for some $A>0$. 
By choosing $f= e^{-\frac12 \|x\|_K^2}$, \eqref{e:FBS} recovers the classical Blaschke--Santal\'{o} inequality  since we have  $v( e^{-\frac12 \|\cdot\|_K^2} ) = {(2\pi)^n}{|\mathbf{B}^n_2|^{-2}} v(K)$ from \eqref{e:Func->Geo}. 
We note that the assumption of the evenness was weakened to the condition that the center of mass of $f$ is 0 in \cite{AKM,CFL,CFM,LehecDirect,LehecYaoYao}.

The functional Blaschke--Santal\'{o} inequality may be stated in the following equivalent way: for any nonnegative and even $f_1,f_2 \in L^1(\mathbb{R}^n)$ satisfying the duality relation 
\begin{equation}\label{e:LegendreDual}
    f_1(x_1) f_2(x_2) \le e^{ - \langle x_1,x_2\rangle },\quad x = (x_1,x_2) \in \mathbb{R}^{2n}, 
\end{equation}
it holds that $\int_{\mathbb{R}^n} f_1\, dx_1\int_{\mathbb{R}^n} f_2\, dx_2 \le (2\pi)^n$. 
Clearly, \eqref{e:LegendreDual} is satisfied for $f_2 = f_1^{\circ}$. 
This formulation of the inequality was found by {\cite{LehecDirect}} first. 
Although this is a simple reformulation of the inequality, it enables us to extend the notion of the Legendre duality relation to multiple input functions. 
Let $m\ge2$ be a natural number. For a tuple of nonnegative functions $\mathbf{f} = (f_1,\ldots, f_m)$, we consider the generalized Legendre duality relation 
\begin{equation}\label{e:KWDuality}
    \prod_{i=1}^m f_i(x_i) \le e^{- \frac{1}{m-1} \sum_{i<j} \langle x_i,x_j\rangle }, \quad x = (x_1,\ldots,x_m) \in (\mathbb{R}^{n})^m. 
\end{equation}
It was Kolesnikov--Werner \cite{KW} who introduced the notion of this generalized Legendre duality relation. 
Such an extension of the duality relation relation \eqref{e:LegendreDual} is  motivated by the barycenter problem of multiple probability measures with respect to the Wasserstein distance. 
As is well-known, the Legendre duality  \eqref{e:LegendreDual} appears in the dual formulation of the Kantorovich problem, that is the Wasserstein distance between two probability measures. 
Similarly the barycenter problem is closely related to the extension of the Kantorovich duality for multiple probability measures; see forthcoming subsection \ref{Section4.2} for more detailed discussion about this perspective. 
Given the notion of the generalized Legendre duality relation, one may wonder if there is any Blaschke--Santal\'{o} type inequality associated with the duality relation. Kolesnikov--Werner \cite{KW} indeed addressed this question, and gave a partial answer as follows: for any nonnegative and \textit{unconditional}\footnote{We say that a function $f$ is unconditional if $f(\varepsilon_1x_1,\ldots,\varepsilon_nx_n) = f(x)$ for any $x\in \mathbb{R}^n$ and $(\varepsilon_1,\ldots,\varepsilon_n)\in \{-1,1\}^n$.} input $\mathbf{f}$ satisfying \eqref{e:KWDuality}, 
\begin{equation}\label{e:KW-Ineq0}
\prod_{i=1}^m \int_{\mathbb{R}^n} f_i\, dx_i \le \big( \int_{\mathbb{R}^n} e^{-\frac12|x|^2}\, dx \big)^m = (2\pi )^{\frac{nm}2}. 
\end{equation}
They made a conjecture that the same inequality\footnote{
To be precise, they considered more general duality relation 
$
\prod_{i=1}^m f_i(x_i) \le \rho\big( \sum_{i<j} \langle x_i,x_j\rangle \big)$, $ (x_1,\ldots,x_m) \in (\mathbb{R}^n)^m $
for a positive non-increasing function on $[0,\infty)$ with $\int_0^\infty \rho(t^2)^\frac1m\, dt <\infty$, and then proved that 
$$
\prod_{i=1}^m \int_{\mathbb{R}^n} f_i\, dx_i \le \bigg( \int_{\mathbb{R}^n} \rho\big( \frac{m(m-1)}2 |u|^2 \big)^\frac1m\, du \bigg)^m,
$$
if all $f_i$ are unconditional. 
Their full conjecture is about the extension of this inequality involving $\rho$ for all even $f_i$. 
Our framework \eqref{e:KWDuality} and \eqref{e:KW-Ineq0} are the special case of $\rho(t) = e^{-\frac{t}{m-1}}$, and thus the problem we address in this paper is also the special case of Kolesnikov--Werner's conjecture. 
By allowing a slight ambiguity, we will still call this special case of the conjecture as Kolesnikov--Werner's conjecture in below. 
} holds true if one weakens the unconditional assumption to the evenness assumption, like the classical Blaschke--Santal\'{o} inequality. 
The argument of Kolesnikov--Werner for the unconditional case is based on Lehec's argument for the direct proof of the functional Blaschke--Santal\'{o} inequality \cite{LehecDirect,LehecYaoYao}. 
Moreover, this unconditional case seems to be the limitation via the simple adaptation of Lehec's argument because of the following reason. 
One of the difficulty of this conjecture comes from the lack of the linear invariance when $m\ge3$, unlikely the case of $m=2$. Indeed, the case of equality is expected to appear only when all $f_i $ are multiplicative of the standard Gaussian. 
This lack of the linear invariance is critical if one tries to adapt Lehec's argument. Roughly speaking, his argument is to find a nice partition of $\mathbb{R}^n$ first, and then appeal to the linear invariance of the inequality to reduce the matter to the unconditional case. 
If the input is unconditional from the beginning, one may skip the step of finding the nice partition, and this is one of the reason why Kolesnikov--Werner managed to establish the unconditional case. 
Therefore, it is seemingly hard
to push further Lehec's argument to settle the conjecture for $m\ge3$, and one would need new idea.

Our aim of this paper is to further generalize the duality relation \eqref{e:KWDuality} and advance the understanding of this generalized duality relation based on a new link to the Brascamp--Lieb theory regarding the multilinear integral functional. 
Let us give our framework. For $m \in \mathbb{N}$, we denote $[m]:= \{ 1,\ldots,m\}$ and take $\mathbf{n}=(n_1,\ldots,n_m) \in \mathbb{N}^m$. Let 
$$
\R^N = \bigoplus_{i=1}^m \R^{n_i} 
$$ 
be an orthogonal decomposition. Clearly $N= \sum_{i=1}^m n_i$. 
We also take exponents $c_1,\ldots,c_m >0$ and write $\mathbf{c} = (c_1,\ldots, c_m)$. 
Finally, let $\mathcal{Q}$ be an arbitrary symmetric matrix on $\mathbb{R}^N$. 
Such a framework is motivated from the Brascamp--Lieb theory as we will explain later with more details. 
We then consider the following generalization of the Legendre duality relation 
\begin{equation}\label{e:GeneDual}
    \prod_{i=1}^m f_i(x_i)^{c_i} \le e^{-\langle x, \mathcal{Q}x\rangle},
    \quad x \in \mathbb{R}^N
\end{equation}
for nonnegative $f_i \in L^1(\mathbb{R}^{n_i})$, $i=1,\ldots,m$.  The model example is the datum\footnote{In this paper, we denote the identity map on $\mathbb{R}^n$ by ${\rm id}_n$. }   
\begin{equation}\label{e:BSData}
m=2, \quad n_1=n_2=n, \quad c_1=c_2=1,\, \quad 
\mathcal{Q} = \frac12 \begin{pmatrix} 0 & {\rm id}_n \\ 
{\rm id}_n & 0\end{pmatrix},
\end{equation} 
for which the relation \eqref{e:GeneDual} coincides with the classical Legendre duality relation \eqref{e:LegendreDual}. 
Similarly, if $m\ge3$ and 
\begin{equation}\label{e:KWData}
    n_1=\cdots = n_m =n,
    \quad
    c_1=\cdots =c_m=1,\quad 
    \mathcal{Q} = 
    \frac1{2(m-1)}
\begin{pmatrix}
0 &  {\rm id}_{n} & \cdots &  {\rm id}_{n} \\
 {\rm id}_{n}  & 0 &\cdots & \vdots \\
\vdots & &  \ddots &  {\rm id}_{n}  \\
 {\rm id}_{n}  & \dots & {\rm id}_{n}  & 0
\end{pmatrix},
\end{equation}
then \eqref{e:GeneDual} corresponds to the duality relation of Kolesnikov--Werner \eqref{e:KWDuality}. 
As the functional Blaschke--Santal\'{o} inequality suggests, we are interested in the best upper bound of $ \prod_{i=1}^m \big( \int_{\mathbb{R}^{n_i}} f_i\, dx_i \big)^{c_i } $ for all nonnegative and \textit{even} $f_i \in  L^1(\mathbb{R}^{n_i})$, $i=1,\ldots,m$, satisfying the duality relation \eqref{e:GeneDual}. 
Our main result is the centered Gaussian saturation phenomenon for this Blaschke--Santal\'{o} type inequality. 

\begin{theorem}\label{t:MainGaussianSaturation}
    Let $m,n_1,\ldots,n_m \in \mathbb{N}$, $c_1,\ldots,c_m>0$, and $\mathcal{Q}$ be a symmetric matrix on $\mathbb{R}^N$.  
    For any nonnegative and even $f_i \in L^1 (\mathbb{R}^{n_i})$, $i=1,\ldots,m$, satisfying \eqref{e:GeneDual}, 
    \begin{equation}\label{e:MainGaussianSaturtion}
        \prod_{i=1}^m \big(\int_{\mathbb{R}^{n_i}} f_i\, dx_i \big)^{c_i}
        \le
        \sup_{A_1,\ldots,A_m} 
        \prod_{i=1}^m \big(\int_{\mathbb{R}^{n_i}} g_{A_i}\, dx_i \big)^{c_i} 
        ,
    \end{equation}
    where the supremum is taken over all $A_i>0$, $i=1,\ldots,m$, such that $(g_{A_1},\ldots, g_{A_m})$ satisfies \eqref{e:GeneDual}. 
\end{theorem}

It is readily to check that $(g_{A_1},\ldots,g_{A_m})$ satisfies \eqref{e:GeneDual} if and only if the matrix 
$
    \sum_{i=1}^m c_i P_i^* A_i P_i - 2\mathcal{Q} 
$
is positive semidefinite. 
Here, we denote the orthogonal projection onto $\mathbb{R}^{n_i}$ by $P_i \colon \R^N \to \R^{n_i}$ for $i \in [m]$. 
Hence, the right-hand side of \eqref{e:MainGaussianSaturtion} may be written as 
\begin{equation}\label{e:GaussConst-KW}
\sup \bigg\{ \prod_{i=1}^m \big( {\rm det}\, 2\pi A_i^{-1} \big)^{\frac{c_i}2}:\; (A_i)_{i=1}^m\; {\rm such\, that} \; \sum_{i=1}^m c_i P_i^* A_i P_i - 2\mathcal{Q}\ge 0  \bigg\}. 
\end{equation}
That is, Theorem \ref{t:MainGaussianSaturation} reduces the problem of identifying the best constant of the inequality into the finite dimensional problem. 
For instance, for the datum \eqref{e:BSData}, we may compute \eqref{e:GaussConst-KW} directly, and it becomes $(2\pi)^n$. 
In this way, we may recover the functional Blaschke--Santal\'{o} inequality from Theorem \ref{t:MainGaussianSaturation}. 
Similarly, if one takes the datum \eqref{e:KWData}, Theorem \ref{t:MainGaussianSaturation} reduces the conjecture of Kolesnikov--Werner into the Gaussian maximization problem. 
In the first nontrivial case $m=3$, we may borrow the recent result by Kalantzopoulos--Saroglou \cite{KS} to establish the conjecture from Theorem \ref{t:MainGaussianSaturation} directly as follows. 
On the one hand, as a special case of \cite[Theorem 1.7]{KS}, Kalantzopoulos--Saroglou observed that \eqref{e:KW-Ineq0} at $m=3$ holds true if $f_1,f_2$ are even and $f_3$ is unconditional. 
On the other hand, Theorem \ref{t:MainGaussianSaturation} reduces the matter to the case $f_1,f_2,f_3$ are all centered Gaussians $g_{A_1},g_{A_2},g_{A_3}$, in which case we may assume\footnote{
This is because of the invariance of the duality \eqref{e:KWDuality} and the inequality \eqref{e:KW-Ineq0} under the common rotation. 
That is, for an orthogonal matrix $U$ that diagonalizes $A_3$, $\widetilde{f_i}:= f_i\circ U$ satisfies \eqref{e:KWDuality} if and only if so does $f_i$. The invariance of the inequality \eqref{e:KW-Ineq0} under the rotation is evident. 
} that $f_3$ is unconditional. 
Thus, the Gaussian constant \eqref{e:GaussConst-KW} with the datum \eqref{e:KWData} and $m=3$ may be controlled by $(2\pi)^{\frac{3n}{2}}$ as we wished. 
That computing the Gaussian constant \eqref{e:GaussConst-KW} with \eqref{e:KWData} and $m\ge4$ is no longer trivial, and requires a substantial work; we refer to the forthcoming Section \ref{Section5} for this point. 
We finally remark that the quantity \eqref{e:GaussConst-KW} is not always finite. 
For instance, if $\mathcal{Q}$ has no positive eigenvalue the right-hand side of \eqref{e:MainGaussianSaturtion} becomes infinite. Indeed, there is no nontrivial upper bound of $\prod_{i=1}^m \int_{\mathbb{R}^{n_i}} f_i\,dx_i$ in such a case. 
Thus, it is necessary to impose that $\mathcal{Q}$ has at least one positive eigenvalue in order to make the inequality \eqref{e:MainGaussianSaturtion} meaningful. 

\subsection{Inverse Brascamp--Lieb inequality under evenness}\label{Section1.2} 
The key of establishing Theorem \ref{t:MainGaussianSaturation} is to  bring the viewpoint of the Brascamp--Lieb theory into the study of the duality relation \eqref{e:GeneDual}.  
It is thus meaningful to give a short introduction about it. 
As for the most general setup, we take $m,n_1,\ldots,n_m, N \in\mathbb{N}$, $c_1,\ldots,c_m\in\mathbb{R}\setminus\{0\}$, linear surjective maps $B_j :\mathbb{R}^N\to \mathbb{R}^{n_i}$, and real-valued symmetric matrix $\mathcal{Q}$. We abbreviate $\mathbf{B}= (B_1,\ldots,B_m)$ and $\mathbf{c} = (c_1,\ldots,c_m)$, and call $(\mathbf{B},\mathbf{c},\mathcal{Q})$  as the Brascamp--Lieb datum. 
For nonnegative (and non-zero) $\mathbf{f} = (f_1,\ldots,f_m) \in L^1(\mathbb{R}^{n_1}) \times \cdots \times L^1(\mathbb{R}^{n_m})$, we define the Brascamp--Lieb functional by 
$$
{\rm BL}( {\bf B}, {\bf c}, \mathcal{Q} ; {\bf f}) 
\coloneqq
\frac{ \int_{\R^N} e^{\langle x, \mathcal{Q} x\rangle} \prod_{i=1}^m f_i(B_ix)^{c_i}\, dx}{ \prod_{i=1}^m \left( \int_{\R^{n_i}} f_i\, dx_i \right)^{c_i} }
\in (0,\infty]. 
$$
Broadly speaking, the (forward) Brascamp--Lieb theory concerns about the best upper bound of ${\rm BL}( {\bf B}, {\bf c}, \mathcal{Q} ; {\bf f})$ for all input $\mathbf{f}$ by fixing the Brascamp--Lieb datum $(\mathbf{B},\mathbf{c}, \mathcal{Q})$. In a similar way, the inverse Brascamp--Lieb inequality\footnote{
{
In addition, there are so-called Barthe's reverse Brascamp--Lieb inequality \cite{Barthe1} that generalizes the Pr\'{e}kopa--Leindler inequality, and its further extension which is so-called the forward--reverse Brascamp--Lieb inequality due to Liu--Courtade--Cuff--Verd\'{u} \cite{LCCV1,LCCV2}. 
}
} concerns about the best lower bound of ${\rm BL}( {\bf B}, {\bf c}, \mathcal{Q} ; {\bf f})$. 
The archetypal example is the sharp Young convolution inequality due to Beckner \cite{Beck} and Brascamp--Lieb \cite{BraLi_Adv}: 
$$
\int_{\mathbb{R}^{2n}}
f_1(x_1)^{c_1} f_2(x_2)^{c_2} f_3(x_1-x_2)^{c_3}\, dx
\le 
{\rm Y}_{\mathbf{c}} \prod_{i=1}^3 \big( \int_{\mathbb{R}^n} f_i\, dx_i \big)^{c_i}, 
$$
where $c_i\in (0,1)$ such that $c_1+c_2+c_3=2$ and ${\rm Y}_{\mathbf{c}}$ is the famous Beckner--Brascamp--Lieb constant. 
When $c_i \in \mathbb{R}\setminus [0,1]$, the inequality may be reversed, and known as the sharp inverse Young inequality \cite{BraLi_Adv}. 
The sharp Young inequality is closely related to the sharp Hausdorff--Young inequality for Fourier transform, where the presence of the Gaussian kernel $\mathcal{Q}$ is relevant\footnote{To be precise, for the Hausdorff--Young inequality, one needs to allow $\mathcal{Q}$ to be a complex-valued matrix.}. 
More generally, Lieb \cite{Lieb} investigated the inequality of the form 
\begin{equation}\label{e:GaussianKernelLieb}
    \int_{\mathbb{R}^{2n}} 
    e^{\langle (x_1,x_2),\mathcal{Q} (x_1,x_2)\rangle} 
    f_1(x_1)^{\frac1p}
    f_2(x_2)^{\frac1{q'}}\, dx_1dx_2 
    \le 
    C 
    \big( \int_{\mathbb{R}^n} f_1\, dx_1 \big)^{\frac1p}
    \big( \int_{\mathbb{R}^n} f_2\, dx_2 \big)^{\frac1{q'}}
\end{equation}
for $p,q\ge1$ and $\mathcal{Q} = \begin{pmatrix} A & D \\ {}^{\rm t} D & B  \end{pmatrix}$ for some $n\times n$ matrices $A,B,D$ that satisfy a certain nondegeneracy condition. 
By the $L^p$-duality, \eqref{e:GaussianKernelLieb} is equivalent to the $L^p$-$L^q$ boundedness of the integral operator with the Gaussian kernel 
\begin{equation}\label{e:LpLq-Gaussian}
\| \mathcal{G}_\mathcal{Q} f \|_{L^q(\mathbb{R}^n)} \le C \| f\|_{L^p(\mathbb{R}^n)},
\quad 
\mathcal{G}_\mathcal{Q}f(x_2):= \int_{\mathbb{R}^n} e^{\langle (x_1,x_2),\mathcal{Q} (x_1,x_2)\rangle} f(x_1)\, dx_1.  
\end{equation}
This inequality for instance contains Nelson's hypercontractivity \cite{Nelson} as an example; see \cite{Lieb,NT,NT2}. 
For the inverse inequality, that is the lower bound of the Brascamp--Lieb functional, we mention the example 
$$
\int_{\mathbb{R}^{2n}}
e^{\langle x_1,x_2\rangle} f_1(x_1)^{\frac1p} f_2(x_2)^{\frac1p}\, dx_1dx_2
\ge 
C \big( \int_{\mathbb{R}^n} f_1\, dx_1 \big)^{\frac1p}
    \big( \int_{\mathbb{R}^n} f_2\, dx_2 \big)^{\frac1{p}}, 
$$
for $ p \in (0,1)$ that may be read as the reverse $L^p$-$L^{p'}$ boundedness of the Laplace transform 
\begin{equation}\label{e:LaplaceIneq}
\| \mathfrak{L} f \|_{L^{p'}(\mathbb{R}^n)}
\ge 
C 
\| f\|_{L^p(\mathbb{R}^n)},
\quad 
\mathfrak{L}f(x_2):= \int_{\mathbb{R}^n} e^{\langle x_1,x_2\rangle} f(x_2)\, dx_1.
\end{equation}
The sharp constant of this Laplace transform bound has been recently identified by authors \cite{NT2} under the evenness assumption, see also the recent work by Cordero-Erausquin--Fradelizi--Langharst \cite{CFL} for more general inputs. 
These examples are only few parts of the much wider range of scope of the Brascamp--Lieb theory.
Applications and perspectives of the theory is so robust and may be found in Harmonic analysis, combinatorics, analytic number theory, convex / differential geometry, probability, stochastic process and statistics, statistical mechanics, information theory, and theoretical computer science; we refer interested readers to references in \cite{BBBCF}. 
Among them, we mention two further examples. 
The first one is the application to convex geometry that was discovered by Ball, where the significance of  so-called the geometric (forward) Brascamp--Lieb inequality was emphasized. 
He exploited the strength of the inequality to study the size of the volume of the section of a convex body, as well as the inequality for the volume ratio that led to the solution to the reverse isoperimetric problem \cite{BallPhd,Ball91,BallJLMS}. 
The second example is about the application to the Fourier restriction theory in Harmonic analysis. 
In the last two decades, Harmonic analysis, and in particular the theory related to so-called the Fourier restriction conjecture, which is one of the most prestigious open problem in this field, experienced several breakthroughs where the perspective of the Brascamp--Lieb theory played a crucial role. 
The pivotal step of this development is based on the invention of the multilinear Fourier restriction / Kakeya inequality by Bennett--Carbery--Tao \cite{BCT}, which may be regarded as a certain stability of Loomis--Whitney inequality under the perturbation of directions. After that, a generalization of the multilinear Fourier restriction and Kakeya inequalities in a spirit of the Brascamp--Lieb inequality has been established by Bennett--Bez--Flock--Lee \cite{BBFL}. 
The appearance of the multilinear Fourier restriction theory sparked series of significant developments, and it culminated at the establishment of the decoupling inequality by Bourgain--Demeter \cite{BD} in the Fourier restriction theory, as well as the resolution of Vinogradov's meanvalue conjecture from analytic number theory by Bourgain--Demeter--Guth \cite{BDG}.

The fundamental result that underpins the whole Brascamp--Lieb theory is the centered Gaussian saturation phenomenon for the forward case discovered by Lieb \cite{Lieb}. 
\begin{theorem}[Lieb \cite{Lieb}]\label{t:Lieb}
    Let $(\mathbf{B},\mathbf{c},\mathcal{Q})$ be the Brascamp--Lieb datum such that $c_i>0$ and $\mathcal{Q}\ge0$. 
    Then for any $f_i \in L^1(\mathbb{R}^{n_i})$, $i=1,\ldots,m$,  
    $$
    \int_{\mathbb{R}^N} e^{\langle x,\mathcal{Q}x\rangle} \prod_{i=1}^m f_i(B_ix)^{c_i}\, dx 
    \le 
    C 
    \prod_{i=1}^m\big( \int_{\mathbb{R}^{n_i}} f_i\, dx_i \big)^{c_i}, 
    $$
    where 
    $$
    C = \sup_{A_1,\ldots,A_m>0} {\rm BL}(\mathbf{B},\mathbf{c},\mathcal{Q}; (g_{A_1},\ldots,g_{A_m}) ). 
    $$
\end{theorem}
An analogous result for the lower bound of the Brascamp--Lieb functional, that we call the inverse Brascamp--Lieb inequality, has been established by Barthe--Wolff \cite{BW}.  
Their result imposed some nondegeneracy condition on the datum. 
In order to describe their nondegeneracy condition, we order $(c_i)_i$ so that $c_1,\ldots, c_{m_+} >0 > c_{m_++1},\ldots, c_m$ for some $0\le m_+\le m$. Correspondingly, let $\mathbf{B}_+:\mathbb{R}^N \ni x \mapsto (B_1x,\ldots, B_{m_+}x) \in \bigoplus_{i=1}^{m_+} \mathbb{R}^{n_i}$. Finally let $s^-(\mathcal{Q})$ denote the number of negative eigenvalues of $\mathcal{Q}$. 
With these notations, Barthe--Wolff \cite{BW} proved that if the datum  $(\mathbf{B},\mathbf{c},\mathcal{Q})$ satisfies the non-degeneracy condition 
    \begin{equation}\label{e:BW-Nondeg}
        \mathcal{Q}|_{{\rm Ker}\, \mathbf{B}_+}<0\;\;\;{\rm and}\;\;\; N\ge s^-(\mathcal{Q}) + \sum_{i=1}^{m_+} n_i, 
    \end{equation}
   then for any $\mathbf{f}$, 
    \begin{equation}\label{e:IBL}
    \int_{\mathbb{R}^N} e^{\langle x,\mathcal{Q}x\rangle} \prod_{i=1}^m f_i(B_ix)^{c_i}\, dx 
    \ge 
    C 
    \prod_{i=1}^m\big( \int_{\mathbb{R}^{n_i}} f_i\, dx_i \big)^{c_i}, 
    \end{equation}
    with 
    $$
    C = \inf_{A_1,\ldots,A_m>0} {\rm BL}(\mathbf{B},\mathbf{c},\mathcal{Q}; (g_{A_1},\ldots,g_{A_m}) ). 
    $$
    Hence, unlikely Lieb's theorem for the forward case, the result of Barthe--Wolff for the inverse case is conditional on the Brascamp--Lieb datum.  
    Moreover it was observed in \cite{BW} and \cite{NT,NT2} that there are examples of Brascamp--Lieb data for which the centered Gaussian saturation property fails to hold, and thus some nondegeneracy condition is needed. 
     In our previous works \cite{NT,NT2}, we realized that the failure of the centered Gaussian saturation for the degenerate datum is closely related to the necessity of some symmetry for the functional Blaschke--Santal\'{o} inequality \eqref{e:FBS}. 
Here is the fundamental observation. 
Let $f$ be a nonnegative $L^1(\mathbb{R}^n)$ function which is the input of the functional Blaschke--Santal\'{o} inequality. 
For each $p>0$ that tends to 0, we take the Brascamp--Lieb datum  
\begin{equation}\label{e:BLdataIBL->BS}
m=2,\quad B_i(x_1,x_2)=x_i\quad (i=1,2),\quad c_1 (p)= c_2(p) = \frac1p,\quad \mathcal{Q}_p= \frac1{2p}  \begin{pmatrix}
	0 & {\rm id}_n \\
	{\rm id}_n & 0
\end{pmatrix},
\end{equation}
and suppose the inverse Brascamp--Lieb inequality 
$$
\int_{\mathbb{R}^{2n}} 
e^{\langle x,\mathcal{Q}_p x\rangle} 
\prod_{i=1,2} f_i(x_i)^{c_i(p)}\, dx 
\ge C_{n,p} \prod_{i=1,2} \big( \int_{\mathbb{R}^n} f_i\, dx_i \big)^{c_i(p)}
$$
for some nontrivial constant $C_{n,p}>0$, and $(f_1,f_2) = (f,f^\circ)$. 
Then we observed in \cite[Lemma 1.3]{NT2} that 
\begin{equation}\label{e:IBL->FBS}
	v(f) \le \lim_{p\to0} C_{n,p}^{-p}. 
\end{equation}
As is well-known, if one drops the assumption on the evenness of $f$, the functional Blaschke--Santal\'{o} inequality fails\footnote{More precisely, to expect the same conclusion as \eqref{e:IBL->FBS}, one needs to normalize the functional volume product by taking the infimum over translations if $f$ is not even; see \cite{FMZSurvey, CFL}. }, and this is one way to understand the failure of the Gaussian saturation phenomenon for the inverse Brascamp--Lieb inequality. 
We emphasize that the datum \eqref{e:BLdataIBL->BS} does not satisfy the  nondegeneracy condition of Barthe--Wolff \eqref{e:BW-Nondeg}. 
In summary, we observed that 
\begin{enumerate}
	\item 
	Some nondegeneracy condition on the Brascamp--Lieb datum such as \eqref{e:BW-Nondeg} is necessary for the validity of the centered Gaussian saturation for the inverse Brascamp--Lieb inequality. 
	\item 
	The direct link between the inverse Brascamp--Lieb inequality and Blaschke--Santal\'{o} inequality appears when the datum is degenerate in the sense of Barthe--Wolff. 
\end{enumerate}
These two observation lead us to the following speculation: \textit{the Gaussian saturation for the inverse Brascamp--Lieb inequality \eqref{e:IBL} would hold true regardless of the Brascamp--Lieb datum if one imposes the evenness on inputs $\mathbf{f}$}. 
Our second main result confirms this speculation for a large family of Brascamp--Lieb data that are relevant to convex geometry under the assumption of the evenness and log-concavity. 
Let us give a setup for this result. 
In below we will consider positive $c_1,\ldots,c_m>0$ only. Also we restrict our attention to the case that $\mathbb{R}^N$ is orthogonally decomposed into $\bigoplus_{i=1}^m \mathbb{R}^{n_i}$; recall the framework of the previous subsection. 
Also the linear map $B_i$ is taken to be the orthogonal projection $P_i$ onto $\mathbb{R}^{n_i}$. 
In this case, $\mathbf{B}$ is uniquely determined by $\mathbf{n}= (n_1,\ldots, n_m)$, and so we use the notation 
$$
{\rm BL}({\bf f})
=
{\rm BL}( {\bf n}, {\bf c}, \mathcal{Q} ; {\bf f}) 
\coloneqq
\frac{ \int_{\R^N} e^{\langle x, \mathcal{Q} x\rangle} \prod_{i=1}^m f_i(x_i)^{c_i}\, dx}{ \prod_{i=1}^m \left( \int_{\R^{n_i}} f_i\, dx_i \right)^{c_i} }
\in (0,\infty], 
$$
where $x=(x_1, \dots, x_m) \in \R^{n_1} \times \cdots \times \R^{n_m}$. 
When $f_i = g_{A_i}$ for some positive definite $A_i$, we denote  
$$
{\rm BL}(\mathbf{A})
= 
{\rm BL}( \mathbf{n},\mathbf{c},\mathcal{Q};\mathbf{A} )
\coloneqq
\frac{ \int_{\R^N} e^{\langle x, \mathcal{Q} x\rangle} \prod_{i=1}^m g_{A_i}(x_i)^{c_i}\, dx}{ \prod_{i=1}^m \left( \int_{\R^{n_i}} g_{A_i}\, dx_i \right)^{c_i} }. 
$$

\begin{theorem}\label{t:MainIBL}
    Let $m,n_1,\ldots,n_m\in \mathbb{N}$, $c_1,\ldots,c_m >0$, and $\mathcal{Q}$ be a symmetric matrix on $\mathbb{R}^N$.  
    For all nonnegative, even and log-concave $f_i \in L^1(\mathbb{R}^{n_i})$, 
    \begin{equation}\label{e:MainIBL}
        \int_{\mathbb{R}^N} e^{\langle x, \mathcal{Q}x\rangle} \prod_{i=1}^m f_i(x_i)^{c_i}\, dx 
        \ge 
        {\rm I}_{\mathbf{g}} (\mathbf{n},\mathbf{c},\mathcal{Q})
        \prod_{i=1}^m 
        \big( \int_{\mathbb{R}^{n_i}} f_i\, dx_i \big)^{c_i},
    \end{equation}
    where 
    $$
    {\rm I}_{\mathbf{g}}(\mathbf{n},\mathbf{c},\mathcal{Q})
    := 
    \inf_{A_1,\ldots,A_m>0} {\rm BL}(\mathbf{A}). 
    $$
\end{theorem}

As the symmetric assumption on $\mathbf{f}$ is essential here, this type of inequality may be referred as the \textit{symmetric inverse Brascamp--Lieb inequality}.
Few remarks are in order. 
Firstly, when the datum $(\mathbf{n},\mathbf{c},\mathcal{Q})$ satisfies the  nondegeneracy condition of Barthe--Wolff \eqref{e:BW-Nondeg}, the inequality \eqref{e:MainIBL} is no more than \eqref{e:IBL}, and thus one does not need to impose the evenness nor log-concavity on $\mathbf{f}$. 
However, as we mentioned in the previous subsection, the most interesting cases from viewpoint of convex geometry is \eqref{e:KWData} in which case the nondegeneracy condition fails to hold. 
Secondly, \eqref{e:MainIBL} has been also established for the specific datum  \eqref{e:BLdataIBL->BS}, which is the degenerate case, in our previous work \cite{NT2}. In there, we only imposed the evenness on $\mathbf{f}$, and so it is reasonable to expect that \eqref{e:MainIBL} would hold without the log-concavity assumption. 
As we will see below, we will appeal to the log-concavity assumption only when we confirm the existence of the extremizer of the symmetric inverse Brascamp--Lieb inequality in Theorem \ref{t:Extremiser}. 
In other words, the proof of Theorem \ref{t:MainIBL} works well whenever one a priori knows the existence of the extremizer. 
From this view point, the notion of the ``amplifying" Brascamp--Lieb datum in \cite{BN} could be useful. 
As other approaches towards this purpose, one may take three possible routes: the one based on heat flow, the one based on mass transport, and the one based on stochastic flow. 
For the first approach, there is one serious difficulty; see the forthcoming discussion after Proposition \ref{t:Monotone*}. 
For the second approach, we do not have any clear evidence to conclude whether it is tractable or not. 
But, if one follows this strategy such as the work of Barthe--Wolff \cite{BW}, one would realize a difficulty of how to exploit the evenness assumption in the mass transport argument. Towards this direction, we mention very recent result by Colesanti--Kolesnikov--Livshyts--Rotem \cite{CKLR} where they gave the mass transport proof of the functional Blaschke--Santal\'{o} inequality based on an observation made in \cite{KW}. 
Perhaps the third option would be the most tractable one. 
This is because of the recent work by Courtade--Fathi--Mikulincer \cite{CFM} where they gave the stochastic proof of the symmetric Talagrand inequality, and thus the functional Blaschke--Santal\'{o} inequality. 
They in fact introduced the novel argument that uses so-called the ``time reversal stochastic flow", which is robust enough to overcome the difficulty appeared in the heat flow argument. 
Finally, if one recalls our introduction of the inverse Brascamp--Lieb inequality \eqref{e:IBL}, one may wonder whether Theorem \ref{t:MainIBL} may be further generalized for the datum involving linear maps $B_i$ and some negative exponents $c_i<0$. 
Apart from some technical justifications, our argument in this paper is in principle applicable to deal with more general Brascamp--Lieb data that involve linear maps $B_j$. 
On the other hand, it is not trivial how one may modify the argument in this paper in order to allow negative exponents $c_i<0$. 

The structure of the paper is as follows. 
In Section \ref{Section2}, we will prove Theorem \ref{t:MainIBL} regarding the symmetric inverse Brascamp--Lieb inequality. 
As a consequence of this symmetric inverse Brascamp--Lieb inequality, we will prove Theorem \ref{t:MainGaussianSaturation} in Section \ref{Section3}. 
Applications to convex geometry, information theory as well as optimal transportation theory will be given in Section \ref{Section4}. 
Section \ref{Section5} will be devoted to the analysis of the Gaussian constant in order to complete the proof of the special case of the conjecture of Kolesnikov--Werner.

\section{Proof of Theorem \ref{t:MainIBL}: symmetric inverse Brascamp--Lieb inequality under log-concavity}\label{Section2}
\subsection{Overview of the proof of Theorem \ref{t:MainIBL} and setup of the regularized framework}\label{Section2.1}
We will work on the regularized framework for the inverse Brascamp--Lieb inequality in order to establish Theorem \ref{t:MainIBL}. 
Let $0 <\lambda < \Lambda<\infty$ be parameters that describe the magnitude of the regularization. 
A function $f=e^{-\varphi} \colon \R^n \to [0, \infty)$ is said to be $\lambda$-uniformly log-concave if $f>0$ on $\mathbb{R}^n$ and 
$$
\varphi( (1-t) x + t y) \le (1-t) \varphi(x) + t \varphi(y) - \frac{\lambda}{2} t(1-t) |x-y|^2, \;\;\; x, y \in \R^n,\,  t \in [0,1]. 
$$
Similarly, a function $f=e^{-\varphi} \colon \R^n \to [0, \infty)$ is said to be $\Lambda$-uniformly log-convex if $f>0$ on $\mathbb{R}^n$ and $\frac1f$ is $(-\Lambda)$-uniformly log-concave. 
We formally extend the definition to the case $\lambda=0$ and $\Lambda=\infty$ by  regarding $0$-uniformly log-concavity as the log-concavity, and $\infty$-uniformly log-convexity as no restriction. 
It is beneficial to introduce a class of even and log-concave functions denoted by 
$$
\mathcal{F}_{LC}^{(e)}(\R^n) \coloneqq \{ f \in L^1(\R^n) \, :\, \text{$f \ge 0$, even  and log-concave {with $\int_{\mathbb{R}^n} f\, dx >0$} 
}
\},
$$
and a class of regularized functions 
$$
\mathcal{F}_{\lambda, \Lambda}^{(e)}(\R^n) \coloneqq \{ f \in \mathcal{F}_{LC}^{(e)}(\R^n) \, :\, \text{$\lambda$-uniformly log-concave,  $\Lambda$-uniformly log-convex}\}.
$$
It is consistent to formally define $\mathcal{F}^{(e)}_{\lambda,\infty}$, for $\lambda>0$, as a set of all positive and even $\lambda$-uniformly log-concave $f$. When $\lambda= 0$, we regard $\mathcal{F}^{(e)}_{0,\infty} = \mathcal{F}^{(e)}_{LC}$ where $f$ is allowed to take the value 0. 
For any Brascamp--Lieb datum $(\mathbf{n},\mathbf{c},\mathcal{Q})$, we describe the inverse Brascamp--Lieb constant in this regularized framework by 
\begin{align*}
{\rm I}_{LC}^{(e)}({\bf n}, {\bf c}, \mathcal{Q}) 
&\coloneqq 
\inf_{ f_i \in \mathcal{F}_{LC}^{(e)}(\R^{n_i}) } {\rm BL}(  {\bf n}, {\bf c}, \mathcal{Q} ; {\bf f}),\\ 
{\rm I}_{\lambda,\Lambda}^{(e)}({\bf n}, {\bf c}, \mathcal{Q}) 
&\coloneqq 
\inf_{ f_i \in \mathcal{F}_{\lambda, \Lambda}^{(e)}(\R^{n_i}) } {\rm BL}(  {\bf n}, {\bf c}, \mathcal{Q} ; {\bf f}) . 
\end{align*}
Finally, by noticing that $g_A \in \mathcal{F}^{(e)}_{\lambda,\Lambda}(\mathbb{R}^n)$ if and only if $\lambda {\rm id}_n \le A \le \Lambda {\rm id}_n$, it is suitable to consider a class of regularized Gaussians denoted by 
$$
\mathcal{G}_{\lambda, \Lambda}(\R^{n})
\coloneqq
\{
A >0 : \lambda {\rm id}_n \le A \le \Lambda {\rm id}_n 
\}. 
$$
The corresponding regularized Gaussian Brascamp--Lieb constant is defined by 
$$
{\rm I}_{\mathcal{G}_{\lambda, \Lambda}}({\bf n}, {\bf c}, \mathcal{Q}) 
\coloneqq 
\inf_{A_i \in \mathcal{G}_{\lambda, \Lambda}(\R^{n_i}) } {\rm BL}({\bf n}, {\bf c}, \mathcal{Q} \, ;\, (A_1,\ldots, A_m)). 
$$
In this terminology, Theorem \ref{t:MainIBL} may be read as 
\begin{equation}\label{e:MainIBLVer2}
    {\rm I}^{(e)}_{LC}(\mathbf{n},\mathbf{c},\mathcal{Q})
    = 
    {\rm I}_{\mathbf{g}}(\mathbf{n},\mathbf{c},\mathcal{Q}). 
\end{equation}
The main benefit to work on this regularized framework is about the existence of the minimizer. 
Because of the compactness, it is clear that the minimizer of ${\rm I}_{\mathcal{G}_{\lambda,\Lambda}}(\mathbf{n},\mathbf{c},\mathcal{Q})$ exists regardless of the Brascamp--Lieb datum, as long as $\lambda>0$ and $\Lambda <\infty$. 
We will also prove that the same is true for ${\rm I}^{(e)}_{\lambda,\Lambda}(\mathbf{n},\mathbf{c},\mathcal{Q})$ in Theorem \ref{t:Extremiser}. 
Thanks to the existence of the minimizer, we will next prove the regularized version of Theorem \ref{t:MainIBL}. 
That is, for any Brascamp--Lieb datum $(\mathbf{n},\mathbf{c},\mathcal{Q})$, 
$$
{\rm I}^{(e)}_{\lambda,\Lambda}(\mathbf{n},\mathbf{c},\mathcal{Q})
=
{\rm I}_{\mathcal{G}_{\lambda,\Lambda}}(\mathbf{n},\mathbf{c},\mathcal{Q}),
$$
holds as long as $\lambda>0$ and $\Lambda<\infty$. 
This is the main part of the proof of Theorem \ref{t:MainIBL}; see forthcoming Theorem \ref{t:RegGaussianSaturation}. 
Finally, we conclude the proof of Theorem \ref{t:MainIBL} by the limiting argument $\lambda\to0$ and $\Lambda\to \infty$. 
The idea of such a regularized framework of the Brascamp--Lieb inequality comes from the work of Bennett--Carbery--Christ--Tao \cite[Corollary 8.15]{BCCT} where they gave an alternative proof of Theorem \ref{t:Lieb} by restricting a class of input functions to $\{ u(1,x): u(0,\cdot) \in L^1(\mathbb{R}^n)  \}$ where $u(t,x)$ is the heat solution with initial data $u(0,\cdot)$. 
By virtue of the regularizing effect of heat flow, $u(1,x)$ earns the log-convexity, and indeed it becomes $\Lambda$-uniformly log-convex for some appropriate $\Lambda>0$; see \cite[Lemma 8.6]{BCCT}. 
This idea of the proof of the Gaussian saturation was further developed for the reverse Brascamp--Lieb inequality by Valdimarsson \cite{Val}, and the inverse Brascamp--Lieb inequality under the nondegeneracy condition of Barthe--Wolff by Bez--Nakamura \cite{BN}.

\subsection{Extremizability}\label{Section2.2}


A goal in this subsection is to show that, in the regularized framework, the symmetric inverse Brascamp--Lieb constant is always attained by some functions.  

\begin{theorem}\label{t:Extremiser}
Let $0<\lambda \le \Lambda < +\infty$. 
Then ${\rm I}_{\lambda, \Lambda}^{(e)}({\bf n}, {\bf c}, \mathcal{Q})$ is extremizable. In other wards, there exists some ${\bf f} \in \mathcal{F}_{\lambda, \Lambda}^{(e)}(\mathbb{R}^n)$ such that ${\rm I}_{\lambda, \Lambda}^{(e)}({\bf n}, {\bf c}, \mathcal{Q}) = {\rm BL}({\bf n}, {\bf c}, \mathcal{Q} ; {\bf f})$ holds true. 
\end{theorem}

To show this result, we will employ the following useful lemma on uniform log-concavity and log-convexity.

\begin{lemma}\label{l:Convexity}
Let $0< \lambda \le \Lambda < +\infty$ and $f=e^{-\varphi} \in \mathcal{F}_{\lambda, \Lambda}^{(e)}(\R^n)$ with $\int_{\R^n} f\, dx=1$. 

\begin{itemize}
\item[(1)] 
It holds that 
$$
\frac n2 \log \frac{2\pi}{\Lambda} \le \varphi(0) \le \frac n2 \log \frac{2\pi}{\lambda}. 
$$

\item[(2)] 
It holds that 
$$
\frac {\lambda}{2} |x|^2 + \frac n2 \log \frac{2\pi}{\Lambda}  \le \varphi(x) \le \frac{\Lambda}{2} |x|^2 + \frac n2 \log \frac{2\pi}{\lambda}, \;\;\; x \in \R^n. 
$$

\item[(3)] Fix $r>0$. Then there exists some $C_{n,r,\lambda, \Lambda}>0$ such that for any $x, y \in [-r, r]^n$, it holds that 
$$
|\varphi(x) - \varphi(y)| \le C_{n,r,\lambda, \Lambda} |x-y|. 
$$

\item[(4)] Fix $r>0$. Then there exists some $C_{n,r,\lambda, \Lambda}>0$ such that it holds that 
$$
\sup_{x \in [-r, r]^n} | \varphi(x)| \le C_{n,r,\lambda, \Lambda}. 
$$
\end{itemize}
\end{lemma}

\begin{proof}
(1) Since $f$ is $\Lambda$-uniformly log-convex and even, it holds that 
$$
\varphi(0) = \varphi( \frac x2 + \frac{-x}{2}) \ge \frac12 \varphi(x) + \frac12 \varphi(-x) - \frac{\Lambda}{8} |x - (-x)|^2
=
\varphi(x) - \frac{\Lambda}{2} |x|^2
$$
for $x\in \R^n$.  
This yields that  
\begin{equation}\label{e:LowerBoundVarPhi}
\varphi(x) \le \frac\Lambda2 |x|^2 + \varphi(0).  
\end{equation}

On the other hand, since $f$ is $\lambda$-uniformly log-concave and even, it holds that for any $x \in \R^n$ and $t \in (0,1)$, 
$$
\varphi(0) = \min_{y \in \R^n} \varphi(y) \le \varphi(t x) \le (1-t) \varphi(0) + t \varphi(x) - \frac {\lambda}{2}t(1-t) |x|^2. 
$$ 
After arranging terms and diving by $t$, this {together with the limit $t\to 0$} implies that 
\begin{equation}\label{e:UpperBoundPhi}
\varphi(x) \ge \frac {\lambda}{2} |x|^2 + \varphi(0). 
\end{equation}
Combining \eqref{e:LowerBoundVarPhi} and \eqref{e:UpperBoundPhi}, we obtain that 
$$
e^{- \frac\Lambda2 |x|^2 - \varphi(0) } \le f(x) \le e^{- \frac \lambda2 |x|^2 - \varphi(0)}. 
$$
Applying $\int_{\R^n}f\, dx=1$, one conclude that  
$$
\left( \frac{2\pi}{\Lambda} \right)^{\frac n2} e^{ - \varphi(0) }
\le
1
\le
\left( \frac{2\pi}{\lambda} \right)^{\frac n2} e^{ - \varphi(0)}, 
$$
which yields the desired assertion. 

(2) The desired assertion immediately follows by combining (1) with \eqref{e:LowerBoundVarPhi} and \eqref{e:UpperBoundPhi}. 

(3) 
Let us fix different points $x, y \in [-r, r]^n$, and put $\ell(t)=(1-t)x + ty \in \mathbb{R}^n$ for $t \in \mathbb{R}$. 
We may take  $t_1>1$ with $\ell(t_1) \in [-2r, 2r]^n \setminus  [-r,r]^n$ and take $t_2>t_1$ with $\frac1{|x-y|}=t_2 - t_1$. Since $\varphi \circ \ell$ is convex, it holds that 
\begin{align*}
\varphi(y) - \varphi(x)
&=
\varphi(\ell(1)) - \varphi(\ell(0))
\le
\frac{\varphi(\ell(t_2)) - \varphi(\ell(t_1))}{t_2 - t_1}
\\
&\le
\frac{\frac \Lambda2 |\ell(t_2)|^2 + \frac n2 \log \frac{2\pi}{\lambda} - \frac n2 \log \frac{2\pi}{\Lambda}}{t_2-t_1}
\\
&\le
\frac{\Lambda |\ell(t_2)- \ell(t_1)|^2 + \Lambda | \ell(t_1)|^2 + \frac n2 \log \frac{2\pi}{\lambda} - \frac n2 \log \frac{2\pi}{\Lambda}}{t_2-t_1}
\\
&=
\frac{\Lambda |t_2-t_1|^2 |y-x|^2 + \Lambda | \ell(t_1)|^2 + \frac n2 \log \frac{2\pi}{\lambda} - \frac n2 \log \frac{2\pi}{\Lambda}}{t_2-t_1}
\\
&\le
C_{n,r, \lambda, \Lambda}|y-x|, 
\end{align*}
where the second inequality follows from (2), and the last inequality follows from $\ell(t_1) \in [-2r, 2r]^n$ and $t_2-t_1 = \frac1{|x-y|}$. 
Similarly we obtain $\varphi(x) - \varphi(y) \le C_{n,r,\lambda,\Lambda} |y-x|$, and thus we conclude the desired assertion. 

(4) This assertion immediately follows from (2). 
\end{proof}


\begin{proof}[Proof of Theorem \ref{t:Extremiser}]
Let ${\bf f}^{(k)} = (f_1^{(k)}, f_2^{(k)}, \dots, f_m^{(k)}) \in \mathcal{F}_{\lambda, \Lambda}^{(e)}(\R^{n_1}) \times \cdots \times\mathcal{F}_{\lambda, \Lambda}^{(e)}(\R^{n_m})$ for $k \in \mathbb{N}$ be a minimizing sequence, namely 
\begin{equation}\label{e:Convergence}
\lim_{k \to \infty} {\rm BL}({\bf f}^{(k)}) = {\rm I}_{\lambda, \Lambda}^{(e)}({\bf n}, {\bf c}, \mathcal{Q}). 
\end{equation}
Without loss of generality, we may suppose that $\int_{\R^{n_i}} f_i^{(k)}\, dx_i=1$ for all $i$ and $k$. 

First note that Lemma \ref{l:Convexity} yields that 
$(f_i^{(k)})_{k \in \mathbb{N}}$ is uniformly bounded and uniformly equicontinuous
\footnote{
An uniform boundedness is a conclusion of Lemma \ref{l:Convexity}(4). 
To see that $(f_i^{(k)})_{k \in \mathbb{N}}$ is uniformly equicontinuous, 
let us take $x, y \in [-1, 1]^n$. Without loss of generality we may suppose that $\varphi_i^{(k)}(x) \le \varphi_i^{(k)}(y)$ for fixed $k$, where $f_i^{(k)}=e^{-\varphi_i^{(k)}}$. 
Then we may apply Lemma \ref{l:Convexity}(3) and (4) to see that 
$$
|f_i^{(k)}(x) - f_i^{(k)}(y)| 
=
e^{-\varphi_i^{(k)}(x)} ( 1 - e^{\varphi_i^{(k)}(x)-\varphi_i^{(k)}(y)})
\le
C_{n, \lambda, \Lambda} | \varphi_i^{(x)}-\varphi_i^{(k)}(y) |
\le 
C_{n, \lambda,\Lambda} |x-y|. 
$$
This means the uniform equicontinuity. 
}
 on $[-1, 1]^n$. 
Thus the Arzel\`{a}--Ascoli theorem and Lemma \ref{l:Convexity} imply that there exists some subsequence $(f_i^{(1_k)})_{k \in \mathbb{N}}$ which uniformly converges to some positive function $F_i^{(1)}$ defined on $[-1, 1]^{n_i}$. 
Also since $(f_i^{(1_k)})_{k \in \mathbb{N}}$ is uniformly bounded and uniformly equicontinuous on $[-2, 2]^{n_i}$, the Arzel\`{a}--Ascoli theorem and Lemma \ref{l:Convexity} imply that there exists some subsubsequence $(f_i^{(2_k)})_{k \in \mathbb{N}}$ which uniformly converges to some positive function $F_i^{(2)}$ defined on $[-2, 2]^{n_i}$. 
Clearly we see that $F_i^{(1)}=F_i^{(2)}$ on $[-1, 1]^{n_i}$. 
Iterating this procedure, we obtain a subsequence $(f_i^{(\ell_k)})_{k \in \mathbb{N}}$ and a positive function $F_i^{(\ell)}$ defined on $[-\ell, \ell]^{n_i}$ for each $\ell \in \mathbb{N}$ satisfying the followings: 
\begin{itemize}
\item[(i)] $(f_i^{((\ell+1)_k)})_{k \in \mathbb{N}}$ is a subsequence of $(f_i^{(\ell_k)})_{k \in \mathbb{N}}$. 

\item[(ii)] $(f_i^{(\ell_k)})_{k \in \mathbb{N}}$ uniformly converges to $F_i^{(\ell)}$ on $[-\ell, \ell]^{n_i}$ as $k \to \infty$.  

\item[(iii)] $F_i^{(\ell +1)} = F_i^{(\ell)}$ on $[-\ell, \ell]^{n_i}$. 
\end{itemize}
Now we define a function $F_i$ on $\R^n$ by 
$F_i(x) \coloneqq F_i^{(\ell)}(x)$ if $x \in [-\ell, \ell]^{n_i}$ for some $\ell \in \mathbb{N}$, which is well-defined. 
Moreover take a subsequence $(f_i^{(k_k)})_{k \in \mathbb{N}}$ of $(f_i^{k})_{k \in \mathbb{N}}$. 
From the construction, we may check that $F_i>0$ on $\mathbb{R}^{n_i}$ and 
$$
\lim_{k \to \infty} f_i^{(k_k)}(x) = F_i(x), \;\;\; x \in \R^{n_i}. 
$$
Since $f_i^{(k_k)}$ is $\lambda$-uniformly log-concave,  for any $x, y \in \R^{n_i}$ and $t \in (0,1)$, 
$$
- \log f_i^{(k_k)}( (1-t) x + ty) \le  - (1-t) \log f_i^{(k_k)}(x) - t \log f_i^{(k_k)}(y) - \frac \lambda2 t(1-t) |x-y|^2, 
$$
and thus tending $k \to \infty$, we obtain that 
$$
- \log F_i( (1-t) x + ty) \le  - (1-t) \log F_i(x) - t \log F_i(y) - \frac \lambda2 t(1-t) |x-y|^2
$$
for any $x, y \in \R^{n_i}$ and $t \in (0,1)$. This  means that $F_i$ is $\lambda$-uniformly log-concave. 
Similarly $F_i$ is also $\Lambda$-uniformly log-convex for $i=1, \dots, m$. 
By the same argument, we may also confirm that $F_i$ is even for all $i=1, \dots, m$. 
Moreover by Lemma \ref{l:Convexity}, Legesgue's convergence theorem yields that 
\begin{equation}\label{e:MassPreservation}
\int_{\R^{n_i}} F_i\, dx_i = \lim_{k \to \infty} \int_{\R^n} f_i^{(k_k)}\, dx_i = 1, \;\;\; i=1, \dots, m. 
\end{equation}
Thus $F_i \in \mathcal{F}_{\lambda, \Lambda}^{(e)}(\R^{n_i})$ for $i=1, \dots, m$. 

Finally, \eqref{e:Convergence}, Fatou's lemma and \eqref{e:MassPreservation} yield that 
\begin{align*}
{\rm I}_{\lambda, \Lambda}^{(e)}({\bf n}, {\bf c}, \mathcal{Q})
=&
\lim_{k \to \infty} \int_{\R^N} e^{\langle x, \mathcal{Q} x\rangle} \prod_{i=1}^m f_i^{(k_k)}(x_i)^{c_i}\, dx
\\
\ge&
\int_{\R^N} e^{\langle x, \mathcal{Q} x\rangle} \prod_{i=1}^m F_i(x_i)^{c_i}\, dx
\ge
{\rm I}_{\lambda, \Lambda}^{(e)}({\bf n}, {\bf c}, \mathcal{Q}), 
\end{align*}
which concludes that 
$$
{\rm I}_{\lambda, \Lambda}^{(e)}({\bf n}, {\bf c}, \mathcal{Q})
=
\int_{\R^N} e^{\langle x, \mathcal{Q} x\rangle} \prod_{i=1}^m F_i(x_i)^{c_i}\, dx. 
$$
\end{proof}

\subsection{Gaussian saturation for the symmetric inverse Brascamp--Lieb inequality under the regularization}\label{Section2.3} 

A main theorem in this subsection is that, in the regularized framework, the symmetric inverse Brascamp--Lieb constant is saturated by centered Gaussians. 

\begin{theorem}\label{t:RegGaussianSaturation}
Let $0 < \lambda \le \Lambda <+\infty$. Then 
$$
{\rm I}_{\lambda, \Lambda}^{(e)} ({\bf n}, {\bf c}, \mathcal{ Q})
=
{\rm I}_{\mathcal{G}_{\lambda, \Lambda}} ({\bf n}, {\bf c}, \mathcal{ Q}). 
$$
\end{theorem}

A key idea to show Theorem \ref{t:RegGaussianSaturation}, and also through this paper, is the following monotonicity of the Brascamp--Lieb functional along the convolution. 

\begin{proposition}\label{t:Monotone*}
For any ${\bf f} \in \mathcal{F}_{\lambda, \Lambda}^{(e)}(\R^{n_1}) \times \cdots \times \mathcal{F}_{\lambda, \Lambda}^{(e)}(\R^{n_m})$, 
$$
{\rm BL}({\bf f})^2 
\ge
{\rm I}_{\lambda, \Lambda}^{(e)} ({\bf n}, {\bf c}, \mathcal{Q}) {\rm BL}( 2^{\frac{n_1}{2}}f_1 \ast f_1(\sqrt{2} \cdot), \dots, 2^{\frac{n_m}{2}}f_m \ast f_m(\sqrt{2} \cdot)). 
$$
\end{proposition}

Before proving this assertion, it is appropriate to mention about Ball's inequality that plays a crucial role in the study of the forward Brascamp--Lieb inequality; see \cite{Barthe1, BCCT}. 
Especially in \cite[Lemma 6.1]{BCCT}, for given measurable ${\bf f}$, ${\bf f'}$ and zero matrix $\mathcal{Q}=O$, it was shown that\footnote{
Precisely, Ball's inequality is established in \cite[Lemma 6.1]{BCCT} associated with linear maps $(B_j)_{j=1}^m$. 
} 
\begin{equation}\label{e:BallIneq}
{\rm BL}({\bf n}, {\bf c}, O; {\bf f}) {\rm BL}({\bf n}, {\bf c}, O; {\bf f'})
\le
{\rm BL}({\bf n}, {\bf c}, O; {\bf f}*{\bf f}')
\sup_{{\bf h}} {\rm BL}({\bf n}, {\bf c}, O; {\bf h}). 
\end{equation}
In particular, by choosing ${\bf f'}$ as any centered Gaussian, this observation provides the evidence that one may expect to prove the forward Brascamp--Lieb inequality by using heat flow. 
Actually, based on this observation, Bennett et al. \cite{BCCT} provides the heat flow proof of the forward Brascamp--Lieb inequality; see also the work of Carlen--Lieb--Loss \cite{CLL}. 
What is an important difference between Theorem \ref{t:Monotone*} and \eqref{e:BallIneq} is whether ${\bf f}={\bf f'}$ or not. 
That we consider the self-convolution $\mathbf{f}\ast \mathbf{f}$ is crucial in the study of the symmetric inverse Brascamp--Lieb inequality as we will see below. 
Moreover, it seems to be less feasible to us to expect the similar monotonicity for the symmetric inverse Brascamp--Lieb constant under $\mathbf{f}\ast \mathbf{f}'$. 
This point is a main difficulty to run the heat flow argument to study the symmetric inverse Brascamp--Lieb inequality. 
Interestingly, the similar difficulty has been observed by Courtade--Fathi--Mikulincer \cite{CFM} in the stochastic proof of the symmetric Talagrand inequality. 
The origin of the idea to consider $ {\rm BL}(\mathbf{f})^2$ may be found in Lieb's original proof in \cite{Lieb}, where he amplified his observation that if $f$ is the one dimensional centered Gaussian, $f(x)f(y) = f( \frac{x+y}{\sqrt{2}} )f( \frac{x-y}{\sqrt{2}} )$. 
In fact, we will also meet such specific structure in the following proof. 
We refer to \cite{LCCV1} for the recent application of this principle in the context of the forward--reverse Brascamp--Lieb inequality by Liu--Courtade--Cuff--Verd\'{u}.

\begin{proof}
We may suppose that $\int_{\R^{n_i}} f_i \, dx_i=1$. 
Put 
$$
F(x) \coloneqq e^{\langle x, \mathcal{Q} x\rangle} \prod_{i=1}^m f_i(x_i)^{c_i}, \;\;\; x=(x_1, \dots, x_m) \in \bigoplus_{i=1}^m \R^{n_i} = \R^N. 
$$
Then 
$$
\int_{\R^N} F \ast F\, dx = {\rm BL}({\bf f})^2. 
$$
On the other hand, 
\begin{align*}
&\int_{\R^N} F \ast F\, dx
\\
=&
2^{\frac{N}{2}} \int_{\R^N} \int_{\R^N} e^{\langle y, \mathcal{Q} y\rangle} \prod_{i=1}^m f_i(y_i)^{c_i} e^{\langle \sqrt{2}x-y, \mathcal{Q} (\sqrt{2}x-y)\rangle} \prod_{i=1}^m f_i(\sqrt{2} x_i - y_i)^{c_i} \, dy\, dx
\\
=&
\int_{\R^N} \int_{\R^N} e^{\langle \frac{x+y}{\sqrt{2}}, \mathcal{Q} \frac{x+y}{\sqrt{2}} \rangle} \prod_{i=1}^m f_i( \frac{x_i + y_i}{\sqrt{2}} )^{c_i} e^{\langle \frac{x-y}{\sqrt{2}}, \mathcal{Q} \frac{x-y}{\sqrt{2}} \rangle} \prod_{i=1}^m f_i( \frac{x_i - y_i}{\sqrt{2}} )^{c_i} \, dy\, dx
\\
=&
\int_{\R^N} e^{\langle x, \mathcal{Q}x \rangle } \int_{\R^N} e^{\langle y, \mathcal{Q}y \rangle } 
\prod_{i=1}^m \left(f_i( \frac{x_i + y_i}{\sqrt{2}} ) f_i( \frac{x_i - y_i}{\sqrt{2}} ) \right)^{c_i} \, dy\, dx, 
\end{align*}
where the first equality follows from changing variables $x$ as $\sqrt{2}x$, and the second equality follows from changing variables $y$ as $\frac{x+ y}{\sqrt{2}}$. 
For fixed $x_i \in \mathbb{R}^{n_i}$, put  
$$
F_i(y_i) = F_i^{(x_i)}(y_i) \coloneqq f_i( \frac{x_i + y_i}{\sqrt{2}} ) f_i( \frac{x_i - y_i}{\sqrt{2}} ), \;\;\; y_i \in \R^{n_i}. 
$$
Then we claim that  $F_i \in \mathcal{F}_{\lambda, \Lambda}^{(e)}(\R^{n_i})$ for $i=1, \dots, m$. 

To see this, first note that $F_i$ is even and $F_i>0$ by definition. 
Next, we put $f_i=e^{-\varphi_i}$. 
Since $f_i$ is $\lambda$-uniformly log-concave, it holds that for any $y_1, y_2 \in \R^{n_i}$ and $t \in (0,1)$, 
\begin{align*}
&\varphi_i \left( \frac{x + (1-t)y_1 + t y_2}{\sqrt{2}} \right)
\\
=&
\varphi_i \left( (1-t) \frac{x + y_1}{\sqrt{2}} + t \frac{x + y_2}{\sqrt{2}} \right)
\\
\le&
(1-t) \varphi_i( \frac{x + y_1}{\sqrt{2}})
+
t \varphi_i( \frac{x + y_2}{\sqrt{2}})
- 
\frac{\lambda}{2} t(1-t) \left| \frac{x + y_2}{\sqrt{2}} - \frac{x + y_1}{\sqrt{2}} \right|^2 
\\
=&
(1-t) \varphi_i( \frac{x + y_1}{\sqrt{2}})
+
t \varphi_i( \frac{x + y_2}{\sqrt{2}})
- 
\frac{\lambda}{4} t(1-t) | y_2-y_1|^2. 
\end{align*}
Similarly, we have that 
$$
\varphi_i \left( \frac{x -( (1-t)y_1 + t y_2)}{\sqrt{2}} \right)
\le
(1-t) \varphi_i( \frac{x - y_1}{\sqrt{2}})
+
t \varphi_i( \frac{x - y_2}{\sqrt{2}})
- 
\frac{\lambda}{4} t(1-t) | y_2-y_1|^2. 
$$
Summing up, we obtain that 
\begin{align*}
&\varphi_i \left( \frac{x + (1-t)y_1 + t y_2}{\sqrt{2}} \right)
+
\varphi_i \left( \frac{x -( (1-t)y_1 + t y_2)}{\sqrt{2}} \right)
\\
\le&
(1-t) \left( \varphi_i( \frac{x + y_1}{\sqrt{2}}) + \varphi_i( \frac{x - y_1}{\sqrt{2}}) \right)
+
t \left( \varphi_i( \frac{x + y_2}{\sqrt{2}}) + \varphi_i( \frac{x - y_2}{\sqrt{2}}) \right)
\\
&
- 
\frac{\lambda}{2} t(1-t) | y_2-y_1|^2, 
\end{align*}
which means that $F_i$ is $\lambda$-uniformly log-concave. 
The similar argument also proves that $F_i$ is $\Lambda$-uniformly log-convex. 
Moreover by $\lambda$-uniformly log-concavity of $F_i$, $F_i \in L^1(\R^{n_i})$ holds true. 
Thus $F_i \in \mathcal{F}_{\lambda, \Lambda}^{(e)}(\R^{n_i})$ for $i=1, \dots, m$. 

By definition, it holds that 
\begin{align*}
&
\int_{\R^N} e^{\langle x, \mathcal{Q}x \rangle } \int_{\R^N} e^{\langle y, \mathcal{Q}y \rangle } 
\prod_{i=1}^m \left(f_i( \frac{x_i + y_i}{\sqrt{2}} ) f_i( \frac{x_i - y_i}{\sqrt{2}} ) \right)^{c_i} \, dy\, dx, 
\\
&\ge
{\rm I}_{\lambda, \Lambda}^{(e)} ({\bf n}, {\bf c}, \mathcal{Q})
\int_{\R^N} e^{\langle x, \mathcal{Q}x \rangle } \prod_{i=1}^m \left( \int_{\R^{n_i}} f_i( \frac{x_i + y_i}{\sqrt{2}} ) f_i( \frac{x_i - y_i}{\sqrt{2}} ) \, dy_i \right)^{c_i}\, dx
\\
&=
{\rm I}_{\lambda, \Lambda}^{(e)} ({\bf n}, {\bf c}, \mathcal{Q})
\int_{\R^N} e^{\langle x, \mathcal{Q}x \rangle } \prod_{i=1}^m \left( 2^{\frac {n_i}{2}} f_i \ast f_i(\sqrt{2} x_i) \right)^{c_i}\, dx
\\
&=
{\rm I}_{\lambda, \Lambda}^{(e)} ({\bf n}, {\bf c}, \mathcal{Q})
{\rm BL}( 2^{\frac{n_1}{2}}f_1 \ast f_1(\sqrt{2} \cdot), \dots, 2^{\frac{n_m}{2}}f_m \ast f_m(\sqrt{2} \cdot)), 
\end{align*}
where the last identity follows from 
$$
\int_{\R^{n_i}} 2^{\frac {n_i}{2}} f_i \ast f_i(\sqrt{2} x_i)\, dx_i
=1, \;\;\; i=1, \dots, m. 
$$
\end{proof}

\begin{proof}[Proof of Theorem \ref{t:RegGaussianSaturation}]
By Theorem \ref{t:Extremiser}, there exists some ${\bf f} \in  \mathcal{F}_{\lambda, \Lambda}^{(e)}(\R^{n_1}) \times \cdots \times\mathcal{F}_{\lambda, \Lambda}^{(e)}(\R^{n_m})$ such that 
$$
{\rm I}_{\lambda, \Lambda}^{(e)} ({\bf c})
=
{\rm BL} ({\bf n}, {\bf c}, \mathcal{Q} ; {\bf f}). 
$$
Without loss of generality, we may suppose that $\int_{\R^{n_i}} f_i \, dx_i=1$. 
Then one may apply Proposition \ref{t:Monotone*} to see that 
$$
{\rm BL}({\bf f})^2 
\ge
{\rm I}_{\lambda, \Lambda}^{(e)} ({\bf n}, {\bf c}, \mathcal{Q}) {\rm BL}( 2^{\frac{n_1}{2}}f_1 \ast f_1(\sqrt{2} \cdot), \dots, 2^{\frac{n_m}{2}}f_m \ast f_m(\sqrt{2} \cdot)). 
$$
As is observed by Brascamp--Lieb \cite{BraLi_JFA}, the uniform log-concavity as well as log-convexity are preserved under the suitably scaled convolution. That is, if $f_i \in\mathcal{F}^{(e)}_{\lambda,\Lambda}(\mathbb{R}^{n_i})$ then  $2^{\frac{n_i}{2}}f_i \ast f_i(\sqrt{2} \cdot) \in \mathcal{F}_{\lambda, \Lambda}^{(e)}(\R^{n_i})$ also holds true; we give a short proof of this fact in Appendix Lemma \ref{l:ConvUniLog}. Thus we may again apply Proposition \ref{t:Monotone*} for $2^{\frac{n_i}{2}}f_i \ast f_i(\sqrt{2} \cdot)$. 
Iterating this procedure, we obtain that 
\begin{align*}
&{\rm I}_{\lambda, \Lambda}^{(e)} ({\bf n}, {\bf c}, \mathcal{Q})^{2^k}
=
{\rm BL} ({\bf f})^{2^k}
\\
&\ge
{\rm I}_{\lambda, \Lambda}^{(e)} ({\bf n}, {\bf c}, \mathcal{Q})^{2^k-1} 
{\rm BL}( (2^k)^{\frac{n_1}{2}} f_1^{(2^k)}({2^{\frac k2}} \cdot), \dots, (2^k)^{\frac{n_m}{2}}f_m^{(2^k)}({2^{\frac k2}} \cdot)), 
\end{align*}
where 
$$
f_i^{(2^k)}
\coloneqq
\overbrace{f_i \ast \cdots \ast f_i}^{\text{$2^k$-times}}, \;\;\; i=1, \dots, m.
$$
Since ${\rm I}_{\lambda, \Lambda}^{(e)} ({\bf n}, {\bf c}, \mathcal{Q})>0$ by Lemma \ref{l:Convexity}, we see that 
$$
{\rm I}_{\lambda, \Lambda}^{(e)} ({\bf n}, {\bf c}, \mathcal{Q})
\ge
{\rm BL}( (2^k)^{\frac{n_1}{2}} f_1^{(2^k)}(2^{\frac k2} \cdot), \dots, (2^k)^{\frac{n_m}{2}}f_m^{(2^k)}(2^{\frac k2} \cdot)). 
$$
On the other hand, the central limit theorem \footnote{For instance, see \cite[Theorem 1.1]{Bob}} yields that there exist some centered Gaussians $g_i$ for each $i=1, \dots, m$ such that 
$(2^k)^{\frac{n_i}{2}} f_i^{(2^k)}(2^{\frac k2} \cdot)$ converges to $g_i$ as $k \to \infty$ in $L^1$ topology, and thus especially pointwisely a.e. on $\R^{n_i}$. 
Since $(2^k)^{\frac{n_i}{2}} f_i^{(2^k)}(2^{\frac k2} \cdot) \in \mathcal{F}_{\lambda, \Lambda}^{(e)}(\R^{n_i})$ and $(2^k)^{\frac{n_i}{2}} \int_{\mathbb{R}^{n_i}} f_i^{(2^k)}(2^{\frac k2} x_i)\, dx_i=1$ for all $k \in \mathbb{N}$, one may also check that $g_i \in \mathcal{G}_{\lambda, \Lambda}(\R^{n_i})$ and $\int_{\mathbb{R}^{n_i}} g_i\, dx_i =1$. 
Hence Fatou's lemma yields that 
\begin{align*}
&\liminf_{k \to \infty} {\rm BL}( (2^k)^{\frac{n_1}{2}} f_1^{(2^k)}(\sqrt{2} \cdot), \dots, (2^k)^{\frac{n_m}{2}}f_m^{(2^k)}(\sqrt{2} \cdot))
\\
&\ge
{\rm BL}( g_1, \dots, g_m)
\ge
{\rm I}_{\mathcal{G}_{\lambda, \Lambda}}({\bf n}, {\bf c}, \mathcal{Q}). 
\end{align*}
It follows from this that 
$$
{\rm I}_{\lambda, \Lambda}^{(e)} ({\bf n}, {\bf c}, \mathcal{Q})
\ge
{\rm I}_{\mathcal{G}_{\lambda, \Lambda}}({\bf n}, {\bf c}, \mathcal{Q}). 
$$

Finally since the opposite inequality 
$$
{\rm I}_{\lambda, \Lambda}^{(e)} ({\bf n}, {\bf c}, \mathcal{Q})
\le
{\rm I}_{\mathcal{G}_{\lambda, \Lambda}}({\bf n}, {\bf c}, \mathcal{Q}) 
$$
is obvious by definition, we conclude the desired assertion. 
\end{proof}

\subsection{Proof of Theorem \ref{t:MainIBL}}\label{Section2.4} 

Let us complete the proof of Theorem \ref{t:MainIBL}. 

\begin{theorem}\label{t:GaussianSaturation}
For any $\lambda\ge0$, 
$$
 {\rm I}_{\lambda, \infty}^{(e)} ({\bf n}, {\bf c}, \mathcal{Q})
=
{\rm I}_{\mathcal{G}_{\lambda, \infty}} ({\bf n}, {\bf c}, \mathcal{Q}). 
$$
Especially, when $\lambda=0$, it holds that 
$$
 {\rm I}_{LC}^{(e)} ({\bf n}, {\bf c}, \mathcal{Q})
=
{\rm I}_{\bf g} ({\bf n}, {\bf c}, \mathcal{Q}). 
$$
\end{theorem}

Obviously Theorem \ref{t:MainIBL} follows from the latter assertion in Theorem \ref{t:GaussianSaturation}.

To show Theorem \ref{t:GaussianSaturation}, let us introduce the intermediate nice class $\mathcal{N}_{\varepsilon_0, \lambda}$ for each $\varepsilon_0>0$ and $\lambda_0\ge0$ defined by 
$$
\mathcal{N}_{\varepsilon_0, \lambda_0}(\mathbb{R}^n)
:= 
\big\{
f \in \mathcal{F}_{\lambda_0, \infty}^{(e)}(\mathbb{R}^n) \; :  \; \exists C_{f}>0\; \text{s.t.}\; f(x)\le C_{f} e^{ -\varepsilon_0 |x|^4 },\;  |x|\ge1 
\big\}. 
$$

\begin{lemma}\label{l:Reduction1}
Let $\lambda_0\ge0$. Then 
	\begin{equation}\label{e:Reduction1}
		{\rm I}_{\lambda_0, \infty}^{(e)} ({\bf n}, {\bf c}, \mathcal{Q})
		= 
		\lim_{\varepsilon_0\to0} 
		\inf_{ f_i \in \mathcal{N}_{\varepsilon_0, \lambda_0}(\R^{n_i}) }
		{\rm BL}(\mathbf{f}).
	\end{equation}
\end{lemma}

\begin{proof}
Take any $f_i \in \mathcal{F}_{\lambda_0,\infty}^{(e)}(\mathbb{R}^{n_i})$, and simply let $f_i^{(\varepsilon_0)}(x):= f_i (x)e^{-\varepsilon_0|x|^4}$. 
Then $f_i^{(\varepsilon_0)}$ is integrable, even and $\lambda_0$-uniformly log-concave. 
Since $f_i$ is even log-concave, $f_i(x_i)\le f_i(0) =: C_{f_i}$ and hence $f_i^{(\varepsilon_0)} \in \mathcal{N}_{\varepsilon_0, \lambda_0}(\mathbb{R}^{n_i})$. 
Thus, thanks to the assumption of $c_i>0$, the monotone convergence theorem ensures that 
\begin{align*}
	{\rm I}_{\lambda_0,\infty}^{(e)} ({\bf n}, {\bf c}, \mathcal{Q})
	=
	\inf_{f_i\in\mathcal{F}_{\lambda_0,\infty}^{(e)}(\mathbb{R}^{n_i})}
	\lim_{\varepsilon_0\to0} 
	\textrm{BL}(\mathbf{f}^{(\varepsilon_0)})
	\ge 
	\lim_{\varepsilon_0\to0} 
		\inf_{ f_i \in \mathcal{N}_{\varepsilon_0,\lambda_0}(\mathbb{R}^{n_i}) }
		{\rm BL}(\mathbf{f}).
\end{align*}
The reverse inequality is evident. 
\end{proof}

\begin{lemma}\label{l:ApproxRegIBL}
Let $\lambda_0 \ge 0$. Then 
	\begin{equation}\label{e:ApproxRegIBL}
		{\rm I}_{\lambda_0, \infty}^{(e)} ({\bf n}, {\bf c}, \mathcal{Q})
		=
		\lim_{\lambda \to 0} \lim_{\Lambda\to \infty} 
		{\rm I}_{\lambda_0+\lambda, \Lambda}^{(e)} ({\bf n}, {\bf c}, \mathcal{Q}).
	\end{equation}
\end{lemma}

\begin{proof}
Fix arbitrary $\varepsilon_0>0$, and take arbitrary $f_i \in \mathcal{N}_{\varepsilon_0, \lambda_0}(\mathbb{R}^{n_i})$. 
For $\lambda>0$ and $\Lambda<+\infty$ with $\lambda +\lambda_0 <\Lambda$, put 
$$
(f_i)_{\lambda}(x_i):= f_i(x_i)e^{-\frac12 \lambda |x_i|^2},\quad  
(f_i)_{ \lambda,\Lambda }(x_i):= e^{\frac1{2\Lambda}\Delta } (f_i)_{\lambda}(x_i).
$$
Note that 
$$
(f_i)_{ \lambda,\Lambda }(x_i)
=
\gamma_{\Lambda^{-1} {\rm id}_{\mathbb{R}^n}} \ast (f_i)_\lambda(x_i)
=
\frac{1}{(2\pi/\Lambda)^{n_i/2}} \int_{\mathbb{R}^{n_i}} e^{ - \frac{\Lambda}2|x_i-y_i|^2} (f_i)_{\lambda}(y_i)\, dy_i, 
$$
and the Li--Yau inequality (or applying Lemma \ref{l:ConvUniLog} in Appendix) provides the gain of log-convexity 
$$
- \nabla^2 \log\, (f_i)_{ \lambda,\Lambda }(x_i) 
\le \Lambda.
$$

Thanks also to the assumption that $f_i$ is $\lambda_0$-uniformly log-concave, $(f_i)_\lambda$ is $(\lambda_0 +\lambda)$-uniformly log-concave, and thus $(f_i)_{\lambda,\Lambda}$ is $((\lambda + \lambda_0)^{-1} +\Lambda^{-1})^{-1}$-uniformly log-concave by applying Lemma \ref{l:ConvUniLog} in Appendix. 
Especially taking large enough $\Lambda \ge \Lambda_{\lambda, \lambda_0}>0$ depending on $\lambda_0$ and $\lambda$, we may suppose that $((\lambda + \lambda_0)^{-1} +\Lambda^{-1})^{-1} \ge \lambda_0 + \frac\lambda2$. 
Moreover since $\int_{\mathbb{R}^{n_i}} f_i\, dx_i>0$, it holds that $f_i \not\equiv 0$ a.e. on $\mathbb{R}^{n_i}$. This means that $(f_i)_{\lambda, \Lambda}>0$ on $\mathbb{R}^{n_i}$. 
Therefore, $(f_i)_{\lambda,\Lambda} \in \mathcal{F}_{\lambda_0+\frac\lambda2,\Lambda}^{(e)}(\mathbb{R}^{n_i})$. 

Let us fix $\lambda>0$, and show that 
\begin{equation}\label{e:LimitChange}
	\lim_{\Lambda\to\infty} \textrm{BL}( (\mathbf{f})_{\lambda,\Lambda} ) = \textrm{BL}( (\mathbf{f})_{\lambda} ). 
\end{equation}
To show this, we claim the following pointwise bound 
\begin{equation}\label{e:PWBoundHeat}
	(f_i)_{\lambda,\Lambda}(x_i)
	\le 
	C_{f_i,n_i} 
	\big(
	e^{-c \varepsilon_0 |x_i|^4}
	+ 
	e^{- c\Lambda|x_i|^2 }
	\big) , 
\end{equation}
for some numerical constant $c>0$. 
For the time being, we assume \eqref{e:PWBoundHeat}, and proceed the proof. 
With \eqref{e:PWBoundHeat} in mind, we take a large $\Lambda_0=\Lambda_0( \mathcal{Q},\mathbf{c}, \lambda, \lambda_0 )> \Lambda_{\lambda, \lambda_0}$ such that 
\begin{equation}\label{e:Lambda_0}
\int_{(\mathbb{R}^n)^m} e^{\langle x, \mathcal{Q}x\rangle} e^{- c\Lambda_0 \sum_{i=1}^m c_i |x_i|^2 }\, dx <+\infty. 
\end{equation}
This is possible since $c_i>0$ and $x = (x_1,\ldots,x_m)$. 
Then we let 
$$
F_i(x_i):= C_{f_i,n_i} 
	\big(
	e^{-c \varepsilon_0 |x_i|^4}
	+ 
	e^{- c\Lambda_0|x_i|^2 }
	\big),
	\quad 
	F(x):= e^{\langle x,\mathcal{Q}x\rangle}
	\prod_{i=1}^m F_i(x_i)^{c_i}. 
$$
We here emphasize that the choice of $F_i,F$ is independent of $\Lambda$. 
On the one hand, $F_i \in L^1(\mathbb{R}^{n_i})$ and $F\in L^1( \mathbb{R}^N )$ by virtue of \eqref{e:Lambda_0}. 
On the other hand, \eqref{e:PWBoundHeat} means that 
$$
\Lambda\ge \Lambda_0
\quad \Rightarrow \quad 
(f_i)_{\lambda,\Lambda} \le F_i,\quad 
e^{\langle x,\mathcal{Q}x\rangle} \prod_{i=1}^m (f_i)_{\lambda, \Lambda}(x_i)^{c_i} \le F(x).
$$
Thus, $F_i,F$ are dominating functions that allows us to apply the Lebesgue convergence theorem to see \eqref{e:LimitChange}. 
Since 
$$
\lim_{\lambda\to 0} \textrm{BL}( (\mathbf{f})_\lambda )
= 
\textrm{BL}(\mathbf{f})
$$
is an easy consequence of the monotone convergence theorem and $c_i>0$, 
we derive that 
$$
\lim_{\lambda\to 0} \lim_{\Lambda\to\infty} \textrm{BL}( (\mathbf{f})_{\lambda,\Lambda} )
= 
\textrm{BL}(\mathbf{f}).
$$
By recalling $(f_i)_{\lambda,\Lambda}\in\mathcal{F}_{\lambda_0+\frac\lambda2,\Lambda}^{(e)}(\mathbb{R}^{n_i})$, this confirms that 
\begin{align*}
\inf_{f_i\in \mathcal{N}_{\varepsilon_0,\lambda_0}(\mathbb{R}^{n_i})} \textrm{BL}(\mathbf{f}) 
&= 
\inf_{f_i\in \mathcal{N}_{\varepsilon_0,\lambda_0}(\mathbb{R}^{n_i})} 
\lim_{\lambda\to 0} \lim_{\Lambda\to\infty} \textrm{BL}( (\mathbf{f})_{\lambda,\Lambda} )
\\
&\ge 
\lim_{\lambda\to 0} \lim_{\Lambda\to\infty} 
{\rm I}_{\lambda_0+\frac\lambda2, \Lambda}^{(e)} ({\bf n}, {\bf c}, \mathcal{Q})
=
\lim_{\lambda\to 0} \lim_{\Lambda\to\infty} 
{\rm I}_{\lambda_0+\lambda, \Lambda}^{(e)} ({\bf n}, {\bf c}, \mathcal{Q}).
\end{align*}
Since this is uniform in $\varepsilon_0$, from Lemma \ref{l:Reduction1}, 
$$
{\rm I}_{\lambda_0,\infty}^{(e)} ({\bf n}, {\bf c}, \mathcal{Q}) \ge \lim_{\lambda\to 0} \lim_{\Lambda\to\infty} 
{\rm I}_{\lambda_0+\lambda, \Lambda}^{(e)} ({\bf n}, {\bf c}, \mathcal{Q}).
$$
This concludes \eqref{e:ApproxRegIBL} since the reverse inequality is evident. 

Let us complete the proof by giving the proof of \eqref{e:PWBoundHeat}. 
Since $f_i\in \mathcal{N}_{\varepsilon_0,\lambda_0}(\mathbb{R}^{n_i})$, we have that $(f_i)_\lambda \le f_i \le C_{f_i} e^{-\varepsilon_0 |x_i|^4}$. 
Thus, 
\begin{align*}
(f_i)_{\lambda,\Lambda}(x_i)
&\le 
C_{f_i} 
\big( \frac{\Lambda}{2\pi} \big)^{\frac{n_i}2}	\int_{\mathbb{R}^{n_i}} e^{-\frac{\Lambda}2|y_i|^2} e^{- \varepsilon_0|x_i-y_i|^4}\, dy_i \\
&= 
C_{f_i,n_i} 
\int_{\mathbb{R}^{n_i}} 
e^{-\frac12|y_i|^2} e^{-\varepsilon_0|x_i-\frac{1}{\sqrt{\Lambda}}y_i|^4}\, dy_i\\
&= 
C_{f_i,n_i} 
\bigg(
\int_{|y_i|\le \frac{\sqrt{\Lambda}}{10}|x_i|}  
e^{-\frac12|y_i|^2} e^{-\varepsilon_0|x_i-\frac{1}{\sqrt{\Lambda}}y_i|^4}\, dy_i\\
&\qquad \qquad 
+
\int_{|y_i|\ge \frac{\sqrt{\Lambda}}{10}|x_i|}
e^{-\frac12|y_i|^2} e^{-\varepsilon_0|x_i-\frac{1}{\sqrt{\Lambda}}y_i|^4}\, dy_i
\bigg).
\end{align*}
For the first term, notice that 
$$
|y_i|\le \frac{\sqrt{\Lambda}}{10}|x_i|
\quad \Rightarrow 
\quad 
\big|x_i-\frac{1}{\sqrt{\Lambda}} y_i\big|
\ge 
\big||x_i|-\frac{1}{\sqrt{\Lambda}} |y_i|\big|
\ge \frac{9}{10} |x_i|.
$$
Thus, 
\begin{align*}
	\int_{|y_i|\le \frac{\sqrt{\Lambda}}{10}|x_i|}  
	e^{-\frac12|y_i|^2} e^{-\varepsilon_0|x_i-\frac{1}{\sqrt{\Lambda}}y_i|^4}\, dy_i
	\le 
    C_{n_i} e^{-(\frac{9}{10})^4\varepsilon_0|x_i|^4}.
\end{align*}
 For the second term, in view of the asymptotic estimate $\int_{K}^\infty e^{-\frac12 t^2}\, dt \sim c \frac{1}{K}e^{- \frac12 K^2}$ as $K\to \infty$, 
\begin{align*}
 	\int_{|y_i|\ge \frac{\sqrt{\Lambda}}{10}|x_i|}
e^{-\frac12|y_i|^2} e^{-\varepsilon_0|x_i-\frac{1}{\sqrt{\Lambda}}y_i|^4}\, dy_i  
&\le C
e^{-c \Lambda |x_i|^2}. 
\end{align*}
These two bounds conclude \eqref{e:PWBoundHeat}. 
\end{proof}

\begin{proof}[Proof of Theorem \ref{t:GaussianSaturation}]
Fix $\lambda_0\ge0$. 
By Lemma \ref{l:ApproxRegIBL} and Theorem \ref{t:RegGaussianSaturation}, we see that 
$$
{\rm I}_{\lambda_0,\infty}^{(e)} ({\bf n}, {\bf c}, \mathcal{Q})
		=
		\lim_{\lambda \downarrow 0} \lim_{\Lambda\to \infty} 
		{\rm I}_{\mathcal{G}_{\lambda_0+\lambda, \Lambda}} ({\bf n}, {\bf c}, \mathcal{Q})
  \ge
  {\rm I}_{\mathcal{G}_{\lambda_0,\infty}} ({\bf n}, {\bf c}, \mathcal{Q}). 
$$
On the other hand, since the opposite inequality is evident, we conclude the desired assertion. 
\end{proof}

\begin{remark}
    Without any difficulty, one may get rid of the assumption of the positivity of $f_i$ in Theorem \ref{t:GaussianSaturation} even when $\lambda>0$, whenever $\Lambda=\infty$. 
    Namely we may also show that 
    \begin{equation}\label{e:NonnegGaussSat}
    \inf_{{\bf f}} {\rm BL}({\bf f})
    =
    {\rm I}_{\lambda, \infty}^{(e)} ({\bf n}, {\bf c}, \mathcal{Q}), 
    \end{equation}
    where the infimum is taken over all {\it nonnegative} $f_i$ which is even, $\lambda$-uniformly log-concave (on its support) and  $0< \int_{\mathbb{R}^{n_i}}\, f_i\, dx_i < \infty$. 
    This follows through the standard approximation argument combining with Theorem \ref{t:GaussianSaturation} and the Lebesgue convergence theorem. 
\end{remark}

\section{Proof of  Theorem \ref{t:MainGaussianSaturation}}\label{Section3}
In this section, we derive Theorem \ref{t:MainGaussianSaturation} from Theorem \ref{t:MainIBL}. 
The basic idea may be found in \eqref{e:IBL->FBS}, that is to consider an appropriate family of Brascamp--Lieb data $(\mathbf{n},\mathbf{c}(p),\mathcal{Q}_p)$, $p>0$, and then take a limit $p\to0$ in the corresponding inverse Brascamp--Lieb inequality. 
To run out this strategy rigorously, we will work in the regularized framework again, and then apply the limiting argument to get rid of the regularization in the end. 
For this purpose, let us introduce further notations. 
Let $(\mathbf{n},\mathbf{c},\mathcal{Q})$ be arbitrary Brascamp--Lieb datum. 
We denote a class of regularized functions that satisfies the generalized Legendre duality relation \eqref{e:GeneDual} by 
$$
\mathcal{D}_{\lambda,\Lambda}(\mathbf{n},\mathbf{c},\mathcal{Q})
:= 
\big\{
\mathbf{f} \in \mathcal{F}^{(e)}_{\lambda,\Lambda}(\mathbb{R}^{n_1})\times \cdots \times \mathcal{F}^{(e)}_{\lambda,\Lambda}(\mathbb{R}^{n_m}):
\prod_{i=1}^m f_i(x_i)^{c_i} \le e^{-\langle x, \mathcal{Q}x\rangle }\big\}
$$
Furthermore, when $\Lambda=\infty$, we allow that the support of ${\bf f} \in \mathcal{D}_{\lambda,\Lambda}(\mathbf{n},\mathbf{c},\mathcal{Q})$ is not the whole space. 
Here, we understand the pointwise inequality holds for any $x = (x_1,\ldots,x_m) \in \mathbb{R}^N$ in the above. 
We also use a notation $\mathcal{D}_{LC}(\mathbf{n},\mathbf{c},\mathcal{Q}) := \mathcal{D}_{0,\infty}(\mathbf{n},\mathbf{c},\mathcal{Q})$. 
The Gaussian analogue is given by 
$$
\mathcal{D}_{\mathbf{g},\lambda,\Lambda}(\mathbf{n},\mathbf{c},\mathcal{Q})
:= 
\big\{
\mathbf{A} : \lambda {\rm id}_{n_i}\le  A_i \le \Lambda {\rm id}_{n_i},\; 
\sum_{i=1}^m c_i P_i^* A_i P_i -2\mathcal{Q} \ge0  
\big\}. 
$$
and 
$ \mathcal{D}_{\mathbf{g}} := \mathcal{D}_{\mathbf{g},0,\infty} $. 

We first fix arbitrary $\mathbf{n},c_1^{(0)},\ldots,c_m^{(0)}>0$ and $\mathcal{Q}^{(0)}$. 
The appropriate family of Brascamp--Lieb data is given by 
$$
c_i=c_i(p):= \frac1{p} ( c_i^{(0)} + p) ,\quad \mathcal{Q}=\mathcal{Q}_p:= \frac1p \mathcal{Q}^{(0)},
$$
for $p>0$. 
The first step is to take a limit $p\to0$ in the inverse Brascamp--Lieb inequality  under frozen $0<\lambda < \Lambda\le\infty$. 
We here emphasize that the case of $\Lambda = \infty$ is allowed in the following arguments. In other words, we will appeal to the benefit of the regularization of $\lambda$ only. 

\begin{lemma}\label{l:p->0BLUpper}
    Fix $0<\lambda < \Lambda \le +\infty$ and let $f_i \in \mathcal{F}_{\lambda,\Lambda}^{(e)}(\mathbb{R}^{n_i})$ for $i=1, \dots, m$. Then 
    \begin{equation}
\limsup_{p\to0} {\rm I}^{(e)}_{\lambda,\Lambda}(\mathbf{n},\mathbf{c}(p),\mathcal{Q}_p)^p 
\prod_{i=1}^m \big( \int_{\mathbb{R}^{n_i}} f_i\, dx_i \big)^{c_i^{(0)}}
\le \sup_{x \in \mathbb{R}^N} e^{\langle x,\mathcal{Q}^{(0)} x\rangle }\prod_{i=1}^m f_i(x_i)^{c_i^{(0)}}.
    \end{equation}
\end{lemma}

\begin{remark}
    In Lemma \ref{l:p->0BLUpper}, when $\Lambda=\infty$, $f_i$ is allowed to take the value $0$ by virtue of \eqref{e:NonnegGaussSat}. 
\end{remark}

\begin{proof}
We may suppose that 
$$
M:= \sup_{ x \in \R^N}  e^{\langle x,\mathcal{Q}^{(0)} x\rangle }\prod_{i=1}^m f_i(x_i)^{c_i^{(0)}} <+\infty,  
$$
otherwise the conclusion is obvious. 
From the argument in Lemma \ref{l:Convexity},  $f_i(x_i) \le f_i(0) e^{-\frac12 \lambda |x_i|^2}$, and hence 
\begin{align*}
F_p(x)
&:= 
\big( e^{\langle x,\mathcal{Q}_px\rangle} \prod_{i=1}^m f_i(x_i)^{c_i(p)} \big)^p 
\le 
M e^{-\frac{\lambda p}2|x|^2} \prod_{i=1}^m f_i(0)^{p} .
\end{align*}
Remark that $0<f_i(0)<+\infty$ since $f_i \in \mathcal{F}_{\lambda,\Lambda}^{(e)}(\mathbb{R}^{n_i})$. 
This confirms that 
\begin{align*}
\bigg(
\int_{\mathbb{R}^N} e^{ \langle x, \mathcal{Q}_px\rangle } \prod_{i=1}^m f_i(x_i)^{c_i(p)}\, dx 
\bigg)^p
&\le 
M \prod_{i=1}^m f_i(0)^{p} \big( \int_{\mathbb{R}^N} e^{-\frac{\lambda }2 |x|^2} \, dx \big)^p \\
&= 
M \prod_{i=1}^m f_i(0)^{p} 
\big( (2\pi \lambda^{-1} )^{\frac{N}2} \big)^p,
\end{align*}
from which we see that 
\begin{equation}\label{e:L-inftyLimit}
	\limsup_{p\to0} \bigg(
\int_{\mathbb{R}^N} e^{ \langle x, \mathcal{Q}_px\rangle } \prod_{i=1}^m f_i(x_i)^{c_i(p)}\, dx 
\bigg)^p
\le 
M = \sup_{ x \in \R^N}  e^{\langle x,\mathcal{Q}^{(0)} x\rangle }\prod_{i=1}^m f_i(x_i)^{c_i^{(0)}}.  
\end{equation}

On the other hand, by definition, we have that 
\begin{align*}
& \bigg(
\int_{\mathbb{R}^N} e^{ \langle x, \mathcal{Q}_px\rangle } \prod_{i=1}^m f_i(x_i)^{c_i(p)}\, dx 
\bigg)^p
\ge 
\textrm{I}_{\lambda,\Lambda}(\mathbf{n},\mathbf{c}(p),\mathcal{Q}_p)^p 
\prod_{i=1}^m \big( \int_{\mathbb{R}^{n_i}} f_i\, dx_i \big)^{c_i^{(0)} + p}. 
\end{align*}
Thus, together with \eqref{e:L-inftyLimit}, we obtain the desired assertion. 
\end{proof}

For the Gaussian constant, we have the following: 
\begin{lemma}\label{l:p->0BLLower}
     Fix $0<\lambda < \Lambda \le +\infty$. Then 
     \begin{align*}
	\liminf_{p\to0} 
	{\rm I}_{\mathcal{G}_{\lambda,\Lambda}}( \mathbf{n},\mathbf{c}(p), \mathcal{Q}_p )^p 
	\ge
	\inf_{\mathbf{A} \in \mathcal{D}_{\mathbf{g},\lambda,\Lambda}(\mathbf{n},\mathbf{c}^{(0)},\mathcal{Q}^{(0)}) } \prod_{i=1}^m \big( {\rm det}\, 2\pi A_i^{-1} \big)^{ -\frac12 c_i^{(0)}  }. 
\end{align*}
\end{lemma}

\begin{proof}
Fix any $A_i>0$ such that $\lambda\, {\rm id}_{n_i} \le A_i \le \Lambda\, {\rm id}_{n_i}$ and 
\begin{equation}\label{e:Assump8/25}
\sum_{i=1}^m c_i(p) P_i^* A_i P_i - 2 \mathcal{Q}_p
=
\frac1p 
\big( \sum_{i=1}^m (c_i^{(0)} +p) P_i^* A_i P_i - 2 \mathcal{Q}^{(0)} \big) > 0.
\end{equation}
Then we compute that 
\begin{align*}
	&\bigg(
	\int_{\mathbb{R}^N} 
	e^{\langle x, \mathcal{Q}_p x\rangle} 
	\prod_{i=1}^m 
	\gamma_{A_i^{-1}}(x_i)^{ c_i(p) }\, dx 
	\bigg)^p\\
	&= 
	\prod_{i=1}^m \big( {\rm det}\, 2\pi A_i^{-1} \big)^{ -\frac12( c_i^{(0)} + p ) } 
	\big( {\rm det}\, A_i^{-1} \big)^{ \frac{p}2 }
	\bigg(
	p^{\frac{N}2} 
	\int_{\mathbb{R}^N} 
	e^{ \langle x , \mathcal{Q}_{\mathbf{A}}^{(0)} x\rangle } 
	\prod_{i=1}^m e^{ -\frac12 ( c_i^{(0)} + p ) |x_i|^2 } \, dx 
	\bigg)^p \\
	&= 
	\prod_{i=1}^m \big( {\rm det}\, 2\pi A_i^{-1} \big)^{ -\frac12 c_i^{(0)}  } 
	(\frac{p}{2\pi})^{\frac{N}2p}  
	\bigg(
	\int_{\mathbb{R}^N} 
	e^{ \langle x , \mathcal{Q}_{\mathbf{A}}^{(0)} x\rangle } 
	\prod_{i=1}^m e^{ -\frac12 ( c_i^{(0)} + p ) |x_i|^2 } \, dx 
	\bigg)^p, 
\end{align*}
where 
$$
\mathcal{Q}_{\mathbf{A}}^{(0)}:= 
{\rm diag}\, (A_1^{-\frac12},\ldots, A_m^{-\frac12})
\mathcal{Q}^{(0)}
{\rm diag}\, (A_1^{-\frac12},\ldots, A_m^{-\frac12}). 
$$
If we decompose $\mathcal{Q}^{(0)} = P_0^* \mathcal{Q}_+^{(0)} P_0 - P_{m+1}^* \mathcal{Q}^{(0)}_- P_{m+1}$, then the assumption $\lambda\, {\rm id}_{n_i} \le A_i \le \Lambda\, {\rm id}_{n_i}$ yields that 
\begin{align*}
\mathcal{Q}_{\mathbf{A}}^{(0)} 
&\ge 
- 
{\rm diag}\, (A_1^{-\frac12},\ldots, A_m^{-\frac12})
P_{m+1}^* \mathcal{Q}^{(0)}_- P_{m+1}
{\rm diag}\, (A_1^{-\frac12},\ldots, A_m^{-\frac12})\\
&\ge 
- \lambda^{-1} \|P_{m+1}^* \mathcal{Q}^{(0)}_- P_{m+1}\|_{\rm op}\, {\rm id}_{\R^N},
\end{align*}
which is uniform in $\mathbf{A}$. 
Now notice that the set of $\mathbf{A}$ satisfying the condition \eqref{e:Assump8/25} is monotone decreasing as $p\to0$. 
Thus, in the limit, we have only to consider $\mathbf{A}$ satisfying  $\sum_i c_i^{(0)} P_i^* A_i^{-1} P_i - 2\mathcal{Q}^{(0)} \ge 0$, see \cite[(76)]{CouLiu}. 
Thus, we conclude the desired assertion. 
\end{proof}

The following result is the regularized version of Theorem \ref{t:MainGaussianSaturation}.

\begin{theorem}\label{t:RegGaussSatKW}
    Fix $0<\lambda < \Lambda \le +\infty$. Then 
    \begin{align*}
\sup_{{\bf f} \in \mathcal{D}_{\lambda, \Lambda}({\bf n}, {\bf c}^{(0)}, \mathcal{Q}^{(0)})} \prod_{i=1}^m \big( \int_{\mathbb{R}^{n_i}} f_i\, dx_i \big)^{c_i^{(0)}}
= 
\sup_{\mathbf{A} \in \mathcal{D}_{\mathbf{g}, \lambda, \Lambda}({\bf n}, {\bf c}^{(0)}, \mathcal{Q}^{(0)})} \prod_{i=1}^m \big( {\rm det}\, 2\pi A_i^{-1} \big)^{ \frac12 c_i^{(0)}  }. 
    \end{align*}
\end{theorem}

\begin{proof}
Lemma \ref{l:p->0BLUpper} immediately yields that 
\begin{align*}
    \liminf_{p\to0} 	\textrm{I}_{\lambda,\Lambda}^{(e)}(\mathbf{n},\mathbf{c}(p),\mathcal{Q}_p)^{-p} 
&\ge
\sup_{{\bf f} \in \mathcal{D}_{\lambda, \Lambda}({\bf n}, {\bf c}^{(0)}, \mathcal{Q}^{(0)})} \prod_{i=1}^m \big( \int_{\mathbb{R}^{n_i}} f_i\, dx_i \big)^{c_i^{(0)}}
\\
&\ge
\sup_{\mathbf{A} \in \mathcal{D}_{\mathbf{g}, \lambda, \Lambda}({\bf n}, {\bf c}^{(0)}, \mathcal{Q}^{(0)}) } \prod_{i=1}^m \big( {\rm det}\, 2\pi A_i^{-1} \big)^{ \frac12 c_i^{(0)}  }. 
\end{align*}
    On the other hand, Lemma \ref{l:p->0BLLower} and Theorems \ref{t:RegGaussianSaturation} and \ref{t:GaussianSaturation} imply that 
    \begin{align*}
    \sup_{\mathbf{A} \in \mathcal{D}_{\mathbf{g}, \lambda, \Lambda}({\bf n}, {\bf c}^{(0)}, \mathcal{Q}^{(0)}) } \prod_{i=1}^m \big( {\rm det}\, 2\pi A_i^{-1} \big)^{ \frac12 c_i^{(0)}  }
    &\ge
    \limsup_{p\to0} 
	\textrm{I}_{\mathcal{G}_{\lambda,\Lambda}}( \mathbf{n},\mathbf{c}(p), \mathcal{Q}_p )^{-p} 
 \\
 &=
    \limsup_{p\to0} 
	\textrm{I}_{\lambda,\Lambda}^{(e)}( \mathbf{n},\mathbf{c}(p), \mathcal{Q}_p )^{-p} .
    \end{align*}
    Our proof is complete. 
\end{proof}

To show Theorem \ref{t:MainGaussianSaturation}, we simply apply Theorem \ref{t:RegGaussSatKW} with  $\Lambda=\infty$ and then take a limit $\lambda\to0$. 

\begin{lemma}\label{Lam->0KW}
    \begin{align*}
&\lim_{\lambda\to0}
\sup_{ {\bf f} \in \mathcal{D}_{\lambda, \infty}({\bf n}, {\bf c}^{(0)}, \mathcal{Q}^{(0)})}
\prod_{i=1}^m \big( \int_{\mathbb{R}^{n_i}} f_i\, dx_i \big)^{c_i^{(0)}}
=
\sup_{ {\bf f} \in \mathcal{D}_{LC}({\bf n}, {\bf c}^{(0)}, \mathcal{Q}^{(0)}) }
\prod_{i=1}^m \big( \int_{\mathbb{R}^{n_i}} f_i\, dx_i \big)^{c_i^{(0)}},
\end{align*}
and
\begin{align*}
&\lim_{\lambda\to0}
\sup_{ \mathbf{A} \in \mathcal{D}_{\mathbf{g}, \lambda, \infty}({\bf n}, {\bf c}^{(0)}, \mathcal{Q}^{(0)}) } \prod_{i=1}^m \big( {\rm det}\,  A_i^{-1} \big)^{ \frac12 c_i^{(0)}  }
=
\sup_{ \mathbf{A} \in \mathcal{D}_{\mathbf{g}}({\bf n}, {\bf c}^{(0)}, \mathcal{Q}^{(0)})}
\prod_{i=1}^m \big( {\rm det}\,  A_i^{-1} \big)^{ \frac12 c_i^{(0)}  }.
\end{align*}
\end{lemma}

\begin{proof}
Since the argument is completely the same, we only show the first identity. 
Let us take integrable, even log-concave $ f_i$ satisfying  $\prod_{i=1}^m f_i(x_i)^{ c_i^{(0)} } \le e^{ - \langle x, \mathcal{Q}^{(0)} x\rangle } $. 
Put $ f_i^{(\lambda)} := e^{-\frac{\lambda}2|x_i|^2} f_i $, then $f_i^{(\lambda)}$ is integrable, even and $\lambda$-uniformly log-concave. 
Hence by using the monotone convergence theorem, 
\begin{align*}
\sup_{ {\bf f} \in \mathcal{D}_{LC}({\bf n}, {\bf c}^{(0)}, \mathcal{Q}^{(0)})  }
\prod_{i=1}^m \big( \int_{\mathbb{R}^{n_i}} f_i\, dx_i \big)^{c_i^{(0)}}
\le 
\lim_{\lambda\to0}
\sup_{ {\bf f} \in \mathcal{D}_{\lambda, \infty}({\bf n}, {\bf c}^{(0)}, \mathcal{Q}^{(0)}) }
\prod_{i=1}^m \big( \int_{\mathbb{R}^{n_i}} f_i\, dx_i \big)^{c_i^{(0)}}. 
\end{align*}
The reverse direction is trivial. 
\end{proof}

\begin{proposition}
    \label{p:MainGaussianSaturationLC}
    For any ${\bf f} \in \mathcal{D}_{LC}({\bf n}, {\bf c}^{(0)}, \mathcal{Q}^{(0)})$, it holds that 
    \begin{equation*}
        \prod_{i=1}^m \big(\int_{\mathbb{R}^{n_i}} f_i\, dx_i \big)^{c_i^{(0)}}
        \le
       \sup_{  \mathbf{A} \in \mathcal{D}_{\mathbf{g}}({\bf n}, {\bf c}^{(0)}, \mathcal{Q}^{(0)})}
\prod_{i=1}^m \big( {\rm det}\, 2\pi A_i^{-1} \big)^{ \frac12 c_i^{(0)}  }. 
    \end{equation*}
\end{proposition}

\begin{proof}
    Combining Lemma \ref{Lam->0KW} and Theorem \ref{t:RegGaussSatKW}, we see that 
    \begin{align*}
    &
    \sup_{ {\bf f} \in \mathcal{D}_{LC}({\bf n}, {\bf c}^{(0)}, \mathcal{Q}^{(0)})  }
\prod_{i=1}^m \big( \int_{\mathbb{R}^{n_i}} f_i\, dx_i \big)^{c_i^{(0)}}
\\
        &=\lim_{\lambda\to0}
\sup_{ {\bf f} \in \mathcal{D}_{\lambda, \infty}({\bf n}, {\bf c}^{(0)}, \mathcal{Q}^{(0)})  }
\prod_{i=1}^m \big( \int_{\mathbb{R}^{n_i}} f_i\, dx_i \big)^{c_i^{(0)}}
\\
&=
\lim_{\lambda\to0}
\sup_{  \mathbf{A} \in \mathcal{D}_{\mathbf{g}, \lambda, \infty}({\bf n}, {\bf c}^{(0)}, \mathcal{Q}^{(0)}) } \prod_{i=1}^m \big( {\rm det}\, 2\pi A_i^{-1} \big)^{ \frac12 c_i^{(0)}  }
\\
&=
\sup_{  \mathbf{A} \in \mathcal{D}_{\mathbf{g}}({\bf n}, {\bf c}^{(0)}, \mathcal{Q}^{(0)})}
\prod_{i=1}^m \big( {\rm det}\, 2\pi A_i^{-1} \big)^{ \frac12 c_i^{(0)}  }.
    \end{align*}
\end{proof}

Finally let us relax the condition of the log-concavity in Proposition \ref{p:MainGaussianSaturationLC} to complete our argument. 
At this stage, there is no reason to specify $c_i^{(0)}$ and $\mathcal{Q}^{(0)}$, and so we use $c_i$ and $\mathcal{Q}$ below. 

\begin{proof}[Proof of Theorem \ref{t:MainGaussianSaturation}]
Let us denote $\mathcal{Q}=(\mathcal{Q}_{ij})_{1\le, i,j \le m}$. We may suppose that $\mathcal{Q}_{ii}\ge0$ for all $i=1, \dots, m$. 
Otherwise there is nothing to prove since 
$$
\sup_{  \mathbf{A} \in \mathcal{D}_{\mathbf{g}}({\bf n}, {\bf c}, \mathcal{Q})}
\prod_{i=1}^m \big( {\rm det}\, 2\pi A_i^{-1} \big)^{ \frac12 c_i  }=+\infty.
$$

Let $f_i \in L^1(\mathbb{R}^{n_i})$ be a nonnegative even function such that $
    \prod_{i=1}^m f_i(x_i)^{ c_i}  \le e^{ - \langle x, \mathcal{Q} x\rangle }.
    $
Let us define functions $F_i \in L^1(\R^{n_i})$ for $i =1, \dots, m$ by induction. 
Put
$$
F_1(x_1)^{c_1} 
\coloneqq 
\inf_{ \substack{x_i \in \R^{n_i} \\ i=2, \dots, m}} 
\frac{e^{- \langle x, \mathcal{Q}x \rangle}
}{
\prod_{i=2}^m f_i(x_i)^{c_i}
}. 
$$
When $F_1, \dots, F_k$ is defined, $F_{k+1}$ is defined as 
$$
F_{k+1}(x_{k+1})^{c_{k+1}} 
\coloneqq 
\inf_{ \substack{x_i \in \R^{n_i} \\ i \in [m] \setminus \{k\}}} 
\frac{e^{- \langle x, \mathcal{Q}x \rangle}
}{
\prod_{i=1}^{k} F_i(x_i)^{c_i} \prod_{i=k+2}^m f_i(x_i)^{c_i}
}. 
$$
Then by definition, for all $i=1, \dots, m$, it holds that $f_i \le F_i$ and 
\begin{equation*}
\prod_{i=1}^m F_i(x_i)^{c_i} \le e^{- \langle x, \mathcal{Q} x \rangle}. 
\end{equation*}
Moreover we may easily check that $F_i$ is even. 
Finally since $\mathcal{Q}_{ii}\ge0$ for all $i$, $F_i$ is log-concave. 
Moreover, by multiplying ${e^{-\varepsilon|x_i|^2}}$ and taking $\varepsilon\to0$ in the end if necessary, we may suppose that $F_i$ is integrable. 
Thus applying Proposition \ref{p:MainGaussianSaturationLC}, we obtain that 
$$
\prod_{i=1}^m \big(\int_{\mathbb{R}^{n_i}} F_i\, dx_i \big)^{c_i}
        \le
       \sup_{  \mathbf{A} \in \mathcal{D}_{\mathbf{g}}({\bf n}, {\bf c}, \mathcal{Q})}
\prod_{i=1}^m \big( {\rm det}\, 2\pi A_i^{-1} \big)^{ \frac12 c_i  }. 
$$
Finally since $f_i \le F_i$, we conclude the desired assertion. 
\end{proof}

\section{Applications}\label{Section4}
\subsection{Applications to convex geometry}\label{Section4.1}
As we have explained in introduction, Theorem \ref{t:MainGaussianSaturation} together with the datum \eqref{e:BSData} rederives the functional Blaschke--Santal\'{o} inequality \eqref{e:FBS}. 
By considering the datum \eqref{e:KWData}, the problem of the conjecture of Kolesnikov--Werner is reduced to the finite dimensional problem, that is to compute the Gaussian constant $\sup_{\mathbf{A} \in \mathcal{D}_{\mathbf{g}}(\mathbf{n},\mathbf{c},\mathcal{Q})} \prod_{i=1}^m {\rm det}\, A_i^{-1}$. 
Although this finite dimensional problem requires a substantial work for $m\ge3$, we manage to give the affirmative answer to the conjecture of Kolesnikov--Werner. 
\begin{theorem}\label{t:KW}
    Let $n\in\mathbb{N}$ and $m\ge2$. 
    For any nonnegative and even $f_i \in L^1(\mathbb{R}^n)$ satisfying 
    $$
    \prod_{i=1}^m f_i(x_i)
    \le 
    e^{-\frac{1}{m-1} \sum_{i<j} \langle x_i,x_j\rangle},
    \quad (x_1,\ldots,x_m) \in (\mathbb{R}^n )^m, 
    $$
    it holds that 
    \begin{equation}\label{e:KWInequality}
    \prod_{i=1}^m \int_{\mathbb{R}^n} f_i\, dx_i \le 
    \big( \int_{\mathbb{R}^n} e^{-\frac12 |x|^2}\, dx \big)^m
    = 
    (2\pi)^{\frac{nm}2}. 
    \end{equation}
\end{theorem}

\begin{remark}
    If one looks at \cite[Proposition 2.1]{KW}, one may realize that the iterative applications of the classical Blaschke--Santal\'{o} inequality implies the upper bound $\prod_{i=1}^m \int_{\mathbb{R}^n_i} f_i\, dx_i \le ( 2\pi (m-1) )^{\frac{mn}2} $. 
    The inequality \eqref{e:KWInequality} improves this trivial inequality by removing the factor $(m-1)^{\frac{mn}2}$. 
\end{remark}

Regarding the proof of this theorem, the main task is the calculation of the Gaussian constant $\sup_{\mathbf{A} \in \mathcal{D}_{\mathbf{g}}(\mathbf{n},\mathbf{c},\mathcal{Q})} \prod_{i=1}^m {\rm det}\, A_i^{-1}$, and this is no longer trivial for $m\ge3$.
We have already seen in the introduction that it is maximized by $A_i = {\rm id}_n$ when $m=3$ thanks to the work of Kalantzopoulos--Saroglou \cite{KS}.
We postpone the proof of the same fact for $m\ge4$ in the next section, and proceed to exhibit further consequences. 
As in \cite[Theorem 5.1]{KW}, by taking $f_i(x_i) = e^{-\frac12\| x_i\|_{K_i}^2}$ for symmetric convex bodies $K_1,\ldots,K_m \subset \mathbb{R}^n$, we may derive the following. 
\begin{corollary}\label{Cor:SetKW}
    Let $n\in \mathbb{N}$ and $m\ge2$. 
    For symmetric convex bodies $K_1,\ldots,K_m \subset \mathbb{R}^n$ satisfying 
    $$
    \sum_{i<j} \langle x_i,x_j\rangle \le \frac{m-1}2 \sum_{i=1}^m \|x_i\|_{K_i}^2,\quad (x_1,\ldots,x_m) \in (\mathbb{R}^n)^m, 
    $$
   it holds that 
   $$
   \prod_{i=1}^m |K_i| \le |\mathbf{B}^n_2|^m.
   $$
\end{corollary}

Next let us recall the definition of the $\lambda$-affine surface area for $\lambda \in \mathbb{R}$ of a convex function, which was introduced in \cite{CFGLSW}. 
Given a convex function $V \colon \mathbb{R}^n \to \mathbb{R} \cup \{\infty\}$, the $\lambda$-affine surface area is defined as
$$
as_\lambda(V) 
\coloneqq
\int_{\Omega_V} e^{(2\lambda -1) V(x) - \lambda \langle x, \nabla V(x) \rangle} ( {\rm det}\, D^2 V(x) )^\lambda\, dx, 
$$
where $\Omega_V \coloneqq {\rm int}\, ( \{ x \in \mathbb{R}^n \, :\, V(x) < +\infty\})$, and $D^2 V$ is the Hessian of $V$ in the sense of Alexandrov \cite{Alex} and Busemann--Feller \cite{BusFe} which exists almost everywhere in $\Omega_V$. 
Especially if $V$ is twice differentiable at $x \in \Omega_V$, then $D^2V(x) = \nabla^2 V(x)$. 
For properties for the $\lambda$-affine surface area, see \cite{CFGLSW}. 
The following result extends the $\lambda$-affine isoperimetric inequality due to \cite{CFGLSW} to multiple functions. 
Note that when the input functions are all unconditional, such an extension of $\lambda$-affine isoperimetric inequality  has been proved in \cite{KW}.

\begin{theorem}\label{t:LambdaAffineIso}
    Let $n \in \mathbb{N}$ and $m \ge 2$. 
    For any even and convex function $V_i \colon \mathbb{R}^n \to \R \cup \{\infty\}$, $i=1, \dots, m$, satisfying 
    $$
    \sum_{i=1}^m V_i(x_i) \ge \frac1{m-1} \sum_{i<j} \langle x_i, x_j \rangle,\;\;\; (x_1, \dots, x_m) \in (\mathbb{R}^n)^m, 
    $$
    it holds that 
    $$
    \prod_{i=1}^m as_\lambda(V_i) \le as_\lambda(\frac12 |\cdot|^2)^m = (2\pi)^{\frac{nm}{2}}, \;\;\; \forall \lambda \in [0, \frac12]. 
    $$
    Moreover for any even and convex function $V_i \colon \mathbb{R}^n \to \R \cup \{\infty\}$, $i=1, \dots, m$, satisfying 
    $$
    \sum_{i=1}^m V_i^*(x_i) \ge \frac1{m-1} \sum_{i<j} \langle x_i, x_j \rangle,\;\;\; (x_1, \dots, x_m) \in (\mathbb{R}^n)^m, 
    $$
    it holds that 
    $$
    \prod_{i=1}^m as_\lambda(V_i) \le as_\lambda(\frac12 |\cdot|^2)^m = (2\pi)^{\frac{nm}{2}}, \;\;\; \forall \lambda \in [\frac12,1]. 
    $$
\end{theorem}
By adapting the argument in \cite{KW}, this result follows from Theorem \ref{t:KW}, so we omit the proof. 
Furthermore when we take $V= \frac12 \| \cdot \|_K^2$ for a symmetric convex body $K \subset \mathbb{R}^n$, it is known in \cite[Theorem 3]{CFGLSW} that for $p \ge 0$, it holds that 
$$
as_\lambda ( \frac12 \| \cdot \|_{K}^2) = \frac{(2\pi)^{\frac n2}}{n | {\bf B}_2^n|} as_p(K), \;\;\; \lambda = \frac{p}{n+p}. 
$$
Here $as_p(K)$ is the $L_p$-affine surface area of a convex body $K \subset \mathbb{R}^n$ given as 
\begin{equation}\label{e:AffineGeo}
    as_p(K)
\coloneqq
\int_{\partial K} \frac{\kappa_K(x)^{\frac{p}{n+p}}}{\langle x, N_K(x) \rangle^{\frac{n(p-1)}{n+p}}}\, d\mu_K(x), 
\end{equation}
where $N_K(x)$ is the outer unit normal vector at $x \in \partial K$, $\mu_K$ is the surface area measure on $\partial K$ and $\kappa_K(x)$ is the Gauss curvature at $x \in \partial K$. 
The $L_p$-affine surface area has originated in \cite{Hug, Lutwak, SchWerAdv}. Especially when $p=1$, $as_1(K)$ is the classical affine surface area. We refer to \cite{SchWerArXiv} and references therein for geometric interpretations and some properties of the affine surface area. 
The following is an immediate conclusion derived from Theorem \ref{t:LambdaAffineIso} by combining with \eqref{e:AffineGeo}. 

\begin{corollary}\label{Cor:AffineIso}
    Let $n\in \mathbb{N}$ and $m\ge2$. 
    For symmetric convex bodies $K_1,\ldots,K_m \subset \mathbb{R}^n$ satisfying 
    $$
    \sum_{i<j} \langle x_i,x_j\rangle \le \frac{m-1}2 \sum_{i=1}^m \|x_i\|_{K_i}^2,\quad (x_1,\ldots,x_m) \in (\mathbb{R}^n)^m, 
    $$
   it holds that 
   $$
   \prod_{i=1}^m as_p(K_i) \le as_p(\mathbf{B}^n_2)^m = n^m |{\bf B}_2^n|^m.
   $$
\end{corollary}

\subsection{Applications to information theory and optimal transportation theory}\label{Section4.2}
As we briefly explained in the introduction, the conjecture of Kolesnikov--Werner \eqref{e:KWInequality} is motivated from the barycenter problem from optimal transportation theory. 
For two probability measures $\mu_1,\mu_2$ on $\mathbb{R}^n$, the Wasserstein distance of these two probability measures is given by 
$$
W_2^2(\mu_1,\mu_2)
:= 
\inf_{ \nu } \int_{\mathbb{R}^n\times \mathbb{R}^n} |x_1-x_2|^2\, d\nu, 
$$
where the infimum is taken over all probability measures $\nu$ on $\mathbb{R}^n\times \mathbb{R}^n$ having $\mu_1,\mu_2$ as marginals in the sense that $\nu(E_1\times \mathbb{R}^n) = \mu_1(E_1)$ and $\nu(\mathbb{R}^n\times E_2) = \mu_2(E_2)$ for measurable $E_1,E_2 \subset \mathbb{R}^n$. 
The Wasserstein distance is actually a distance function on the set of all probability measures with the finite second moment on $\mathbb{R}^n$, denoted by $\mathcal{P}_2(\mathbb{R}^n)$, and enables us to regard this as a geometric object. 
The pair $(\mathcal{P}_2(\mathbb{R}^n), W_2)$ is called as the $L^2$-Wasserstein space on $\mathbb{R}^n$. 
Generally speaking, it is of interest to investigate geometric properties of metric measure spaces not only from purely mathematical curiosity but also from the view point of applied science; see \cite{CutDou,PeyCut,San} for instance. 
For the $L^2$-Wassersiten space, it is an important property that  $(\mathcal{P}_2(\mathbb{R}^n), W_2)$ is geodesic. 
Moreover if  $\mu_1, \mu_2 \in (\mathcal{P}_2(\mathbb{R}^n), W_2)$ are absolutely continuous with respect to the $n$-dimensional Lebesgue measure, then there exists the unique geodesic between $\mu_1$ and $\mu_2$. 
This property enables us to find the unique midpoint of $\mu_1,\mu_2 \in \mathcal{P}_2(\mathbb{R}^n)$ by considering the minimizer of 
\begin{equation}\label{e:Midpoint}
\mathcal{P}_2(\mathbb{R}^n) \ni \mu \mapsto \frac12 \big( W_2^2(\mu_1,\mu) + W_2^2(\mu_2,\mu) \big). 
\end{equation}
Such an interpolation of two probability measures is known as McCann's interpolation \cite{McCann} that leads the notion of the famous displacement of convexity of functionals on the Wasserstein space. The notion of the displacement of convexity plays a fundamental role in the theory of gradient flows in Wasserstein spaces \cite{AGS}, synthetic theory of the Ricci curvature (the curvature-dimension condition) on metric measure spaces \cite{V}. 
In this context, the functional Blaschke--Santal\'{o} inequality provides a  quantitative comparison between the information of the midpoint of $\mu_1,\mu_2$ and the relative entropy of initial probability measures. 
For an absolutely continuous probability measure $\mu = \rho dx$, its entropy (relative to the standard Gaussian) is defined by 
$$
{\rm H}(\mu|\gamma):= \int_{\mathbb{R}^n} \frac{\rho}{\gamma} \log\, \frac{\rho}{\gamma}\, d\gamma, 
$$
where $\gamma(x):= (2\pi)^{-\frac{n}{2}} e^{-\frac12 |x|^2}$, and otherwise put ${\rm H}(\mu |\gamma) \coloneqq +\infty$. 
Both $W_2^2(\cdot,\gamma)$ and ${\rm H}(\cdot |\gamma)$ measure how close is the probability measure to the standard Gaussian. 
Talagrand's transportation-cost inequality gives the quantitative comparison of these two measurements, and states that 
$
\frac12 W_2^2( \mu, \gamma)
\le 
{\rm H}( \mu|\gamma ). 
$
When either $\mu_1$ or $\mu_2$ is even, a stronger inequality holds true 
$$
    \frac12 W_2^2(\mu_1,\mu_2) 
    \le 
    {\rm H}(\mu_1|\gamma)
    + 
    {\rm H}(\mu_2|\gamma),
$$
which has been established by Fathi \cite{Fathi}. In fact, he observed that this symmetric Talagrand inequality is equivalent to the functional Blaschke--Santal\'{o} inequality. We also mention the work by Tsuji \cite{Tsuji} for its direct transport proof on 1-dimension and its further generalizations for weighted measures under certain restrictions. 
The connection to the midpoint problem \eqref{e:Midpoint} for $\mu_1,\mu_2$ may be found in \cite[Remark 7.2]{KW} by observing that $ \frac12 W_2^2(\mu_1,\mu_2) = W_2^2(\mu_1,\mu) + W_2^2(\mu_2,\mu) $, where $\mu$ is the minimizer of \eqref{e:Midpoint}.  Thus the symmetric Talagrand inequality may be read as 
\begin{equation}\label{e:SymTal}
    W_2^2(\mu_1,\mu) + W_2^2(\mu_2,\mu) \le 
    {\rm H}(\mu_1|\gamma)
    + 
    {\rm H}(\mu_2|\gamma). 
\end{equation}
It was Agueh--Carlier \cite{AC} who extended the notion of the midpoint of two inputs $\mu_1,\mu_2$ \eqref{e:Midpoint} to multiple inputs, that is the barycenter. 
For $\mu_1,\ldots,\mu_m \in \mathcal{P}_2(\mathbb{R}^n)$, the ($L^2$-Wasserstein) barycenter of these measures with weights $\frac1m$ is defined by the solution to the following minimization problem 
\begin{equation}\label{e:Barycenter}
    \inf_{\mu \in \mathcal{P}_2(\mathbb{R}^n)} \frac1{m} \sum_{i=1}^m W_2^2(\mu_i,\mu). 
\end{equation}
Agueh--Carlier \cite{AC} investigated the existence and uniqueness of the minimizer of \eqref{e:Barycenter}. 
Their main motivation to introduce and study the barycenter problem is as follows. On the one hand, it is the generalization of the notion of the midpoint of two probability measures for which there is a beautiful and powerful theory. 
On the other hand, it gives an interesting example of the metric space that confirms the existence and uniqueness of the barycenter \eqref{e:Barycenter} and that is outside of the nonpositively curved space, so-called ${\rm CAT}(0)$ space, at the same time. 
We note that the notion of the barycenter as a minimizer of an averaged squared distance has already been investigated by Sturm \cite{Strum} in the framework of nonpositively curved metric spaces. 
Thus, Agueh--Carlier's motivation was based on purely mathematical curiosity. Nevertheless, the importance of the notion of the barycenter of probability measures in  theoretical computer science has been realized recently, and we refer the survey article by Peyr\'{e}--Cuturi \cite{PeyCut} and references their in. 
These recent development and realization of the importance of the barycenter problem led Kolesnikov--Werner \cite{KW} to the following question: what is a Talagrand's transportation-cost type inequality that captures the information of the barycenter of $\mu_1,\ldots, \mu_m$, and generalizes \eqref{e:SymTal}.  
They proposed the following barycentric Talagrand inequality 
\begin{equation}\label{e:BarycenterTalagrand} 
    \frac1{2m} \sum_{i=1}^m W_2^2(\mu_i,\mu)
    \le 
    \frac{m-1}{m^2} \sum_{i=1}^m {\rm H}(\mu_i|\gamma), 
\end{equation}
where $\mu$ is the barycentetr of $\mu_1,\ldots,\mu_m$, and confirmed it when all $\mu_1,\ldots,\mu_m \in \mathcal{P}_2(\mathbb{R}^n)$ are unconditional and absolutely continuous. 
As for their conjecture on the generalized Blaschke--Santal\'{o} type inequality, they also expected that \eqref{e:BarycenterTalagrand} would hold true even if one weakens the unconditional assumption to  the evenness assumption. 
We confirmed their expectation as follows. 
\begin{theorem}\label{t:BarycenterTal}
    Let $n\in\mathbb{N}$ and $m\ge 2$. 
    Then for any symmetric $\mu_1,\ldots, \mu_m \in \mathcal{P}_2(\mathbb{R}^n)$, \eqref{e:BarycenterTalagrand} holds true. 
\end{theorem}
Once we established Theorem \ref{t:KW}, the proof of Theorem \ref{t:BarycenterTal} is a routine adaptation of the argument in \cite[Proposition 7.3]{KW}, and so we omit the proof.

\section{Analysis of the Gaussian constant: the conjecture of kolesnikov--Werner}\label{Section5}
\subsection{Overview and setup}\label{Section5.1}
In this section we concern about the datum \eqref{e:KWData}. 
As we explained in introduction, this datum corresponds to the conjecture of Kolesnikov--Werner. 
By virtue of Theorem \ref{t:MainGaussianSaturation}, we have that 
$$
\prod_{i=1}^m \int_{\mathbb{R}^n} f_i\, dx_i 
\le 
\sup \big\{
\prod_{i=1}^m \big( {\rm det}\, 2\pi A_i^{-1} \big)^\frac12: 
\mathbf{A}\; {\rm such\, that}\, M(\mathbf{A}) \ge0 
\big\}
$$
for any even $f_i$ satisfying the duality relation  \eqref{e:GeneDual}. 
Here, 
$$
M(\mathbf{A}):= \sum_{i=1}^m P_i^* A_i P_i -2\mathcal{Q}
= 
\begin{pmatrix}
	A_1 & -\frac1{m-1} {\rm id}_n & \cdots & -\frac1{m-1}{\rm id}_n \\
	-\frac1{m-1}{\rm id}_n & A_2 & \cdots & \vdots \\
	\vdots & & \ddots & -\frac1{m-1} {\rm id}_n \\
	-\frac1{m-1} {\rm id}_n & \cdots & -\frac1{m-1}{\rm id}_n & A_m
\end{pmatrix}.
$$
Hence, the problem of the conjecture of Kolesnikov--Werner is now reduced to the finite dimensional problem. 
That is, we are interested in the maximization problem 
\begin{equation}\label{e:MaximizeGaussian-KW}
	{\rm KW}_{\mathbf{g}}:= \sup \big\{ \prod_{i=1}^m {\rm det}\, A_i^{-1}:\; M(\mathbf{A}) \ge 0 \big\},
\end{equation}
and our goal is to show that this is maximized by $A_i = {\rm id}_n$. 

It is worth to giving the entropic formulation of this problem. 
We here use the slightly abusing notation for centered Gaussians: for a symmetric positive definite matrix $A$, 
$$
\gamma_A(x):= \big( {\rm det}\, 2\pi A \big)^{-\frac12} e^{- \frac12 \langle x, A^{-1} x\rangle } = \frac{g_{A^{-1}}}{m( g_{A^{-1}} )}. 
$$
Then the maximization problem \eqref{e:MaximizeGaussian-KW} is equivalent to find the minimizer of 
\begin{equation}\label{e:MinimizeGaussian-Tal}
    {\rm T}_{\mathbf{g}} := \inf \big\{ \mathcal{T}(\gamma_{A_1},\ldots, \gamma_{A_m}):\; A_1,\ldots, A_m >0 \big\},
\end{equation}
where 
$$
\mathcal{T}(\mu_1,\ldots,\mu_m)
:= 
\frac{m-1}{m^2} \sum_{i=1}^m {\rm H}(\mu_i|\gamma)
- 
\frac1{2m} \sum_{i=1}^m W_2^2(\mu_i,\mu)
$$
denotes the deficit of the barycenteric Talagrand inequality. 
In below, we will use $g_{A_i}$ for the functional formulation of the inequality \eqref{e:KWInequality}, and $\gamma_{A_i}$ for the entropic formulation of the inequality \eqref{e:BarycenterTalagrand}.
Note that the constant ${\rm T}_{\mathbf{g}}$ is known to be finite, and more precisely $ {\rm T}_{\mathbf{g}} \in (-\infty,0] $; see Proposition 2.1 in \cite{KW}. 
Our goal here is indeed to show that ${\rm T}_{\mathbf{g}} = 0$.  
One of the benefit of the formulation \eqref{e:MinimizeGaussian-Tal} is about the existence of the extremizer of the inequality. 
The worst scenario of proving the existence of the extremizer often comes from the possibility of the concentration of mass. 
That is, we need to exclude the potential that the minimizing sequence $\mathbf{A}^{(R)}= (A_i^{(R)})_{i=1}^m$ has the following property: at least one of $A_i^{(R)}$ is going to be degenerate as $R\to \infty$, and thus $\gamma_{A_i^{(R)}}$ contains Dirac delta in the limit. 
However, in general, if the input measure contains Dirac delta, its entropy diverges while Wasserstein distance (to other Gaussian measure) stays to be finite. Thus, it should not be the minimizing sequence of $\mathcal{T}$ unless there is some scale invariance structure such as the case of $m=2$. 
These discussions are of course heuristic, 
and we will give more detailed proof in the end of this section. 

It is thus pivotal to show that any extremizer is indeed given by $A_i = {\rm id}_n$ when $m\ge3$. 
In order to show this, we will take two steps. 
\begin{itemize}
    \item[(Step 1)]
    The first step is to extract information about the maximizer of \eqref{e:MaximizeGaussian-KW} as much as possible by following the strategy proposed by Kolesnikov--Werner in the end of their paper \cite{KW}; see Lemma \ref{l:GaussianStep1}. 
    This is enough to conclude that ${\rm KW}_{\mathbf{g}} = 1$ when $m=2$ which corresponds to the classical Blaschke--Santal\'{o} inequality. However, it seems to be highly nontrivial to prove the same conclusion for $m\ge3$. 
    \item [(Step 2)]
    Because of such a difficulty in the case $m\ge3$, we address the following purely linear algebraic problem: 
    given any symmetric positive definite matrices $X_1,\ldots,X_m$ satisfying 
    $$
    0< X_i < {\rm id}_n,\quad \sum_{i=1}^m X_i = {\rm id}_n,\quad X_i-X_i^2 = X_j -X_j^2
    $$
    for any $i,j =1,\ldots, m$, is it true that $X_i$ is necessarily $\frac1m {\rm id}_n$? 
    When $n=1$, this is clearly true by solving equations elementally. 
    We will prove that this is indeed the case for any $n\ge 2$ in Proposition \ref{t:GaussianStep2}. 
\end{itemize}

Before giving detailed proofs, let us give a remark about the difficulty of applying the standard strategy to this Gaussian analysis. 
Given the viewpoint of the inverse Brascamp--Lieb inequality, one may wonder whether is it possible to identify 
${\rm I}_{\mathbf{g}}(\mathbf{n},\mathbf{c}(p), \mathcal{Q}_p)$, 
where 
$$
n_i=n,\quad c_i(p) = \frac1p, \quad \mathcal{Q}_p = \frac{\mu_p}{p} \mathcal{Q} 
$$
for some parameter $\mu_p <1$ such that $\lim_{p\to0} \mu_p = 1$. 
In fact, it is tractable to expect that ${\rm I}_{\mathbf{g}}( \mathbf{n},\mathbf{c}(p),\mathcal{Q}_p )$ is attained by the standard Gaussian $g(x_i) = e^{-\frac12 |x_i|^2}$ for all $p\ll1$ by choosing the suitable parameter $\mu_p$.
When $m=2$, this has been confirmed by authors in \cite{NT2}, and indeed it gave the sharp bound of the Laplace transform \eqref{e:LaplaceIneq}. 
If one could show this, then the desired conclusion ${\rm KW}_{\mathbf{g}}=1$ would immediately follow by taking the limit $p\to 0$; recall Theorem \ref{t:RegGaussSatKW}. 
The standard strategy to identify the Gaussian Brascamp--Lieb constant, which is usually possible for so-called the geometric data, relies on the certain log-convexity/concavity of the functional $\mathbf{A} \mapsto {\rm BL}(\mathbf{A})$; see \cite[Proposition 6]{Barthe1} for the forward Brascamp--Lieb inequality and \cite[Proposition 4.4]{BW} for the inverse Brascamp--Lieb inequality. 
In particular, Theorem 4.5 due to Barthe--Wolff \cite{BW}, which is the consequence of their Proposition 4.4, is of relevant to us. 
It in fact provides the characterization of the Gaussian extremizer of the inverse Brascamp--Lieb inequality for any data that satisfy their nondegeneracy condition. 
So our problem identifying the above ${\rm I}_{\mathbf{g}}(\mathbf{n},\mathbf{c}(p),\mathcal{Q}_p)$ may be regarded as an extension of Theorem 4.5 of Barthe--Wolff to degenerate Brascamp--Lieb data. 
Their proof of Theorem 4.5 heavily depend on their nondegeneracy condition, which enables them to apply some log-concavity type property, and thus such extension seems to require some new idea.

\subsection{Proof of ${\rm KW}_{\mathbf{g}}=1$}\label{Section5.2}
\begin{lemma}\label{l:GaussianStep1}
	Let $n\in\mathbb{N}$ and $m\ge2$. 
	Suppose that $\mathbf{A} = (A_i)_{i=1}^m$ is the maximizer of \eqref{e:MaximizeGaussian-KW}. 
	\begin{enumerate}
		\item 
		It holds that 
		\begin{align}\label{e:A_iA_j}
			&\big( \frac{m-1}m A_i + \frac1m {\rm id}_n \big) \big( \frac{m-1}{m} {\rm id}_n + \frac1m A_i^{-1} \big)\\
			&= 
			\big( \frac{m-1}m A_j + \frac1m {\rm id}_n \big) \big( \frac{m-1}{m} {\rm id}_n + \frac1m A_j^{-1} \big)\nonumber 
		\end{align}
		for any $i,j \in [m]$. 
		\item 
		The barycenter of $\mu_i = \frac{g_{A_i}}{m(g_{A_i})}dx$, $i=1,\ldots,m$, is given by $\frac{g_{A_0}}{m(g_{A_0})} dx$ where 
		\begin{equation}\label{e:BarycenterMaximizer}
		\quad A_0:= \big( \frac{m-1}{m} A_i + \frac{1}{m} {\rm id}_n \big)^{-1} \big( \frac{m-1}{m} {\rm id}_n + \frac{1}{m} A_i^{-1} \big)^{-1}.
		\end{equation}
		Remark that the definition of $A_0$ is independent of the choice of $i$ thanks to \eqref{e:A_iA_j}. 
		\item 
		It holds that 
		\begin{equation}\label{e:MaximizerAverage}
			\sum_{i=1}^m \big( (m-1)A_i + {\rm id}_n \big)^{-1} 
			= {\rm id}_n. 
		\end{equation}
	\end{enumerate} 
\end{lemma}

\begin{proof}
	The proof is to work out the strategy that has been proposed by Kolesnikov--Werner \cite{KW} for Gaussian inputs.  
	First of all, let us recall the direct relation between the barycenter problem and the Kantorovich duality under general setting. 
	The following general properties may be found in Theorem 2.4 in \cite{KW}. 
	Let $\mu_i$ be absolutely continuous measures with finite second moments and $\mu$ be the barycenter of $(\mu_i)_{i=1}^m$. 
	Then it holds that 
	\begin{equation}\label{e:Bary-Kantoro}
	\frac1{2m} \sum_{i=1}^m W_2^2(\mu_i,\mu)
	= 
	\frac{1}{2m^2} \int_{(\mathbb{R}^n)^m} \sum_{i<j} |x_i-x_j|^2\, d\pi.  
	\end{equation}
	Here, $\pi$ is the unique solution to the multimarginal Kantorovich problem 
	$$
	\sup_{ \nu } \int_{(\mathbb{R}^{n})^m} \sum_{i<j} \langle x_i,x_j\rangle\, d\nu, 
	$$
	where the supremum is taken over all measures $\nu$ on $(\mathbb{R}^n)^m$ having $\mu_1,\ldots,\mu_m$ as marginals. 
    This direct relation has been found by Agueh--Carlier \cite{AC}, see also the statement 2 in \cite[Theorem 2.4]{KW}. 
	The duality to the multimarginal Kantorovich problem has been established by Gangbo--\'{S}wiech \cite{GS}, see also the statement 4 in \cite[Theorem 2.4]{KW}, and states that 
	$$
	\frac1{m-1} \int_{(\mathbb{R}^n)^m} \sum_{i<j} |x_i-x_j|^2\, d\pi
	=
	\inf_{W_1,\ldots, W_m} \sum_{ i=1 }^m \int_{\mathbb{R}^n} W_i\, d\mu_i,
	$$
	where the infimum is taken over all $W_1,\ldots, W_m$ such that 
	$$
	\sum_{i=1}^m W_i(x_i) \ge \frac1{m-1} \sum_{i<j} \langle x_i,x_j\rangle ,\quad x_i,x_j \in \mathbb{R}^n.
	$$
	Again, the existence and uniqueness of the dual problem is known. 
	That is, there exists a tuple of convex functions $U_1,\ldots, U_m$ such that 
	\begin{equation}\label{e:MinimizerU_i}
		\frac1{m-1} \int_{(\mathbb{R}^n)^m} \sum_{i<j} |x_i-x_j|^2\, d\pi
		=
		\inf_{W_1,\ldots, W_m} \sum_{ i=1 }^m \int_{\mathbb{R}^n} W_i\, d\mu_i 
		= 
		\sum_{ i=1 }^m \int_{\mathbb{R}^n} U_i\, d\mu_i, 
	\end{equation}
	and that 
	\begin{equation}\label{e:Assump-KWConj}
		\sum_{i=1}^m U_i(x_i) \ge \frac1{m-1} \sum_{i<j} \langle x_i,x_j\rangle ,\quad x_i,x_j \in \mathbb{R}^n. 
	\end{equation}
	Moreover, by denoting the optimal transport map from $\mu$ onto $\mu_i$ by $\nabla \Phi_i$, \eqref{e:Assump-KWConj} becomes equality for $(x_1,\ldots,x_m) \in \{ ( \nabla \Phi_1(y),\ldots, \nabla\Phi_m(y) ): y \in \mathbb{R}^n \}$. 
	In particular, \eqref{e:Bary-Kantoro} and \eqref{e:MinimizerU_i} yield that 
	\begin{equation}\label{e:Bary-DualKantoro}
		\frac1{2m} \sum_{i=1}^m W_2^2(\mu_i,\mu)
		= 
		\frac{m-1}{2m^2} 
		\sum_{i=1}^m \int_{\mathbb{R}^n} U_i\, d\mu_i. 
	\end{equation}
	It is also worth to mentioning that $\Phi_i$ and $U_i$ are constrained by  
	\begin{equation}\label{e:Phi_i-U_i}
		\Phi_i^*(x_i)
		= 
		\frac1{2m} |x_i|^2 + \frac{m-1}{m} U_i(x_i), 
	\end{equation}
	where $\Phi_i^*$ is the Legendre dual of $\Phi_i$. 
	Finally, note that $\Phi_1,\ldots \Phi_m$ satisfy that 
	\begin{equation}\label{e:AverageOTMap}
	\frac1m \sum_{i=1}^m \nabla \Phi_i = {\rm id}_n. 
	\end{equation}

	We now take specific $\mu_i = \frac{g_{A_i}}{ m( g_{A_i} ) }$ where $\mathbf{A}$ is the maximizer of \eqref{e:MaximizeGaussian-KW}. 
	Let $U_i$ be the solution to the dual multimarginal Kantorovich problem for such $\mu_i$. 
	Then $U_i$ is also a quadratic function\footnote{This is because, the fact that all $\mu_i$ are centered Gaussian confirms that the barycenter $\mu$ is also centered Gaussian, and so $\Phi_i$ a quadratic function. 
	Thus, \eqref{e:Phi_i-U_i} yields that $U_i$ is also quadratic. 
	}.
	On the one hand, from Theorem 8.2 in \cite{KW}, it holds that 
	$$
	\prod_{i=1}^m \int_{\mathbb{R}^n} g_{A_i}\, dx_i
	\le 
	\prod_{i=1}^m \int_{\mathbb{R}^n} e^{-U_i}\, dx_i. 
	$$
	On the other hand, by recalling that $e^{-U_i}$ is the centered Gaussian and that $U_i$ satisfies \eqref{e:Assump-KWConj} which is equivalent to the assumption \eqref{e:GeneDual}, the fact that $\mathbf{A}_i$ is the maximizer of \eqref{e:MaximizeGaussian-KW} yields that 
	$$
	\prod_{i=1}^m \int_{\mathbb{R}^n} g_{A_i}\, dx_i
	\ge 
	\prod_{i=1}^m \int_{\mathbb{R}^n} e^{-U_i}\, dx_i. 
	$$
	Thus, $g_{A_1},\ldots,g_{A_m}$ establishes equality of Theorem 8.2 in \cite{KW}. 
	By investigating the equality case of Theorem 8.2 in \cite{KW} from its proof, one may conclude that 
	$$
	U_i(x_i) = \frac12 \langle x_i , A_i x_i\rangle. 
	$$ 
	We refer to the discussion just after Theorem 8.2 in \cite{KW} for further details of this point. 
	From this and \eqref{e:Phi_i-U_i}, we may identify $\Phi_i$ as 
	$$
	\Phi_i(x_i) 
	= 
	\frac12 \langle x_i, \big( \frac{m-1}m A_i + \frac{1}m{\rm id}_n  \big)^{-1} x_i\rangle. 
	$$
	Since $\nabla \Phi_i$ is the optimal transport map from $\mu$ to $\mu_i$, the density $\rho$ of $\mu$ must satisfy that 
	$$
	\rho(x)
	= 
	\frac{g_{A_i}(\nabla \Phi_i(x)) }{ m(g_{A_i}) } {\rm det}\, \nabla^2 \Phi_i(x). 
	$$
	Since this identity holds for all $i=1,\ldots,m$, by inserting the explicit form of $\Phi_i$, we obtain \eqref{e:A_iA_j} and \eqref{e:BarycenterMaximizer}. 
	The identity \eqref{e:MaximizerAverage} follows from the explicit form of $\Phi_i$ together with \eqref{e:AverageOTMap}. 
\end{proof}

When the classical case $m=2$, Lemma \ref{l:GaussianStep1} is enough to conclude that the Gaussian constant \eqref{e:MaximizeGaussian-KW} is achieved for $(A_1,A_2) = (A,A^{-1})$ for any $A>0$ and so ${\rm KW}_{\mathbf{g}}=1$, which is exactly the extremizer of the functional Blaschke--Santal\'{o} inequality. 
Therefore, one may expect that the strategy of Kolesnikov--Werner for the alternative proof of the functional Blaschke--Santal\'{o} inequality would work well when $m=2$. 
Indeed it is very recent that Colesanti--Kolesnikov--Livshyts--Rotem \cite{CKLR} gave the mass transport proof of the functional Blaschke--Santal\'{o} inequality based on this strategy. 
However, the case for $m\ge3$ is a different story, and requires further substantial work. 
In fact, it is no longer trivial to conclude that $A_i = {\rm id}_n$ is the maximizer of \eqref{e:MaximizeGaussian-KW} from Lemma \ref{l:GaussianStep1} as far as we are aware, and proving this fact relies on a nontrivial linear algebraic argument. 

\begin{proposition}\label{t:GaussianStep2}
	Let $n\in \mathbb{N}$ and $m\ge3$. 
	If $\mathbf{A}= (A_i)_{i=1}^m$ is the maximizer of \eqref{e:MaximizeGaussian-KW} then $A_i = {\rm id}_n$, $i=1,\ldots,m$. 
\end{proposition}

\begin{proof}
We denote $X_i:= \big( (m-1) A_i + {\rm id}_n \big)^{-1}$ and reformulate conditions \eqref{e:A_iA_j} and \eqref{e:MaximizerAverage} as follows: 
$$
X_i - X_i^2 = X_j - X_j^2, \;\;\; \forall i,j \in [m]
\quad 
{\rm and} 
\quad 
\sum_{i=1}^m X_i = {\rm id}_n.
$$
Moreover we know that $0<X_i<{\rm id}_{n}$. 
Hence our goal is to show the following: if symmetric positive definite matrices $X_1, \dots, X_m$, where $m\ge3$, satisfies 
\begin{itemize}
\item[(i)]  $0<X_i<{\rm id_{n}}$, 
\item[(ii)] 
$$
\sum_{i=1}^m X_i = {\rm id}_{n},
$$
\item[(iii)]
$$
X_i - X_i^2 = X_j - X_j^2, \;\;\; \forall i,j \in [m], 
$$
\end{itemize}
then 
\begin{equation*}
X_i = \frac1m {\rm id}_{n}, \;\;\; \forall i \in [m].
\end{equation*}
With this terminology, the barycenter $A_0$ is given by 
$$
A_0 = \frac{m^2}{m-1} (X_i - X_i^2),
$$
where the right-hand side is independent of the choice of $i$. 
As a first reduction, we may suppose that $A_0$ is diagonal. 
This is because, the problem is invariant under the common rotation $A_i \mapsto U^* A_i U$, where $U$ is any orthogonal matrix, and hence we may assume $A_1$ is diagonal for instance. 
As a second reduction, we may further assume that $A_0$ is a scaler matrix. 
To see this, let us suppose that the claim is true when $A_0$ is a scalar matrix for the time being, and conclude the claim when $A_0$ is any diagonal matrix. 
The $A_0$ may be written as 
$$
A_0 = {\rm diag}\, ( \underbrace{\lambda_1,\ldots, \lambda_1}_{\mu_1},\ldots, \underbrace{\lambda_l,\ldots,\lambda_l}_{\mu_l}),
$$
for some $l \in \{1,\ldots,n\}$, $\lambda_1,\ldots,\lambda_l>0$, and $\mu_1,\ldots,\mu_l$ denote the multiplicity of each eigenvalues; $\mu_1+\cdots + \mu_l = n$. 
If $l=1$ then $A_0$ becomes a scalar matrix, so suppose $l\ge2$. 
In this case, $A_0$ may be decomposed into $l$ many small block matrices. 
Accordingly, for any $i=1,\ldots,m$, $X_i$ is also decomposed into $l$ many small block matrices: 
$$
X_i = 
\begin{pmatrix}
	X_i^{(1)} & 0 & \cdots & 0 \\
	0 & X_i^{(2)} & \cdots &0 \\
	\vdots & & \ddots & \vdots \\
	0& \cdots & & X_i^{(l)},
\end{pmatrix}
$$
where $X_i^{(k)}$, $k=1,\ldots,l$, is $\mu_k\times \mu_k$ symmetric matrix. 
To see this, one has only to notice the following: if we denote the eigenspace of the eigenvalue $\lambda_k$ of $A_0 = \frac{m^2}{m-1} (X_i - X_i^2)$ by $V_k \simeq \mathbb{R}^{\mu_k}$, then $X_i$ maps $V_k$ to $V_k$. 
Thanks to the block matrix structure of $X_i$, properties (i),(ii),(iii) of $X_i$ infer to each $X_i^{(k)}$, $k=1,\ldots l$. 
Since each $X_i^{(k)}$ is scalar, we may apply the assumption to conclude that $X_i^{(k)} = \frac1m {\rm id}_{\mu_k}$, and hence $X_i = \frac1m {\rm id}_n$. 

As $A_0$ is a scalar matrix, we may have that 
\begin{equation}\label{e:Scalar}
X_i - X_i^2 = \lambda\, {\rm id}_{n}, 
\end{equation}
where $\lambda \in \R$ is independent of $i$. 
First notice that eigenvalues of $X_i$, $i=1,\ldots m$, are given by $\alpha,\ldots, \alpha, 1-\alpha,\ldots, 1-\alpha$ for some $\alpha \in (0,\frac12)$. Here $\alpha$ is independent of $i$, but the multiplicity of $\alpha$ may depend on $i$. 
This is because \eqref{e:Scalar} implies that if $\beta \in (0,1)$ is an eigenvalue of $X_i$, it must satisfy that 
$$
\beta - \beta^2 = \lambda, 
$$
and thus 
$$
\beta =\frac{1 \pm \sqrt{1 - 4\lambda}}{2}, \;\;\; 0 < \lambda \le \frac14.
$$
By denoting $\alpha \coloneqq \frac{1 - \sqrt{1 - 4\lambda}}{2} \in (0, \frac12]$, the eigenvalues of $X_i$ must be $\alpha$ or $1-\alpha$. 
If $\alpha=\frac12$, then all eigenvalues of $X_i$ are only $\frac12$, and hence we see that $X_i = \frac12 {\rm id}_{n}$. 
However, this is prohibited by (ii) and $m\ge3$. 

Secondly, let us show that 
$$
{\rm det}\, X_i = {\rm det}\, X_j, \;\;\; \forall i, j \in [m]. 
$$
Indeed it follows from $m \ge 3$, (ii) and (iii) that 
\begin{align*}
X_i \sum_{k \neq i, j} X_k
&=
X_i ( {\rm id}_n - X_i - X_j)
=
X_i - X_i^2 - X_i X_j
=
X_j - X_j^2 - X_iX_j
\\
&=
({\rm id}_n - X_j - X_i) X_j
=
(\sum_{k \neq i, j} X_k) X_j. 
\end{align*}
Since $ \sum_{k \neq i, j} X_k$ is invertible by $m \ge 3$, we conclude that $
{\rm det}\, X_i = {\rm det}\, X_j. 
$

From this fact and $\alpha \in (0,\frac12)$, we may derive the following:
if we denote the multiplicity of the eigenvalue $\alpha$ of $X_i$ by $k_i \in \{ 0, 1, \dots, n\}$, then $k_i = k_j$ for all $i, j$. 
In what follows, we denote $k=k_i$, which is independent of $i$. 
Moreover, we observe that $k > ( 1 - \frac1m) n$ as follows. 
It follows from (ii) that 
$$
mk \alpha + m(n-k) (1-\alpha)
=
{\rm Tr}\, \sum_{i=1}^m X_i 
=
n, 
$$
which yields that 
$$
(2mk - mn) \alpha = n - mn +mk.
$$
If $2k=n$, then $0= n - mn+mk = (2-m)k$ which is a contradiction. 
Hence $2k \neq n$, and thus 
$$
\alpha = \frac{ n - mn + mk}{2mk-mn}. 
$$
Since $0 < \alpha < \frac12$, we have that 
$$
0 < \frac{ n - mn + mk}{2mk-mn} < \frac12. 
$$
If $2k < n$, then 
$$
0 > n - mn + mk > \frac12 ( 2mk -mn), 
$$
which contradicts with $m\ge3$, and so we always have that $2k>n$. 
This yields that 
$$
0 < n - mn + mk < \frac12 ( 2mk -mn), 
$$
from which we obtain the desired lower bound of $k$. 

The next thing to observe is the identity 
\begin{equation}\label{e:XiEll}
X_i^\ell = \frac1{1-2\alpha} \left( (1-\alpha)^\ell - \alpha^\ell \right) X_i - \frac{\alpha - \alpha^2}{1-2\alpha} \left( (1-\alpha)^{\ell-1} - \alpha^{\ell-1} \right) {\rm id}_{n}
\end{equation}
for all $\ell \in \mathbb{N}$ and any $i$. 
This identity follows by the induction. 
When $\ell=1$, it is evident. 
Suppose \eqref{e:XiEll}, then we may apply \eqref{e:Scalar} to remove $X_i^2$ and see that 
\begin{align*}
X_i^{\ell +1}
&=
\frac1{1-2\alpha} \left( (1-\alpha)^\ell - \alpha^\ell \right) X_i^2 - \frac{\alpha - \alpha^2}{1-2\alpha} \left( (1-\alpha)^{\ell-1} - \alpha^{\ell-1} \right) X_i
\\
&=
\frac1{1-2\alpha} \left( (1-\alpha)^{\ell+1} - \alpha^{\ell+1} \right) X_i - \frac{\alpha - \alpha^2}{1-2\alpha} \left( (1-\alpha)^{\ell} - \alpha^{\ell} \right) {\rm id}_{n}. 
\end{align*}
This identity together with (ii) revels that 
$$
\sum_{i=1}^m X_i^\ell = \frac1{1-2\alpha} \left( (1-\alpha)^\ell - \alpha^\ell \right) {\rm id}_{n} - \frac{\alpha - \alpha^2}{1-2\alpha} \left( (1-\alpha)^{\ell-1} - \alpha^{\ell-1} \right) m\, {\rm id}_{n}. 
$$ 
Regarding the left-hand side, in order to compute $X_i^l$ in terms of $\alpha$, let us decompose $X_i$ as 
$$
X_i = \alpha \sum_{j=1}^k u_{ij} \otimes u_{ij} + (1-\alpha) \sum_{j=k+1}^n u_{ij} \otimes u_{ij}, 
$$
where $(u_{ij})_{1\le j \le n} \in \R^n$ are orthonormal eigenvectors. 
Then 
$$
X_i^\ell = \alpha^\ell \sum_{j=1}^k u_{ij} \otimes u_{ij} + (1-\alpha)^\ell \sum_{j=k+1}^n u_{ij} \otimes u_{ij}, 
$$
and thus 
$$
\sum_{i=1}^m X_i^\ell
=
\alpha^\ell \sum_{i=1}^m \sum_{j=1}^k u_{ij} \otimes u_{ij} + (1-\alpha)^\ell \sum_{i=1}^m \sum_{j=k+1}^n u_{ij} \otimes u_{ij}. 
$$
Combining above two identities and then dividing by $(1-\alpha)^\ell$, it follows that 
\begin{align*}
&
\frac{\alpha^\ell}{(1-\alpha)^\ell} \sum_{i=1}^m \sum_{j=1}^k u_{ij} \otimes u_{ij} + \sum_{i=1}^m \sum_{j=k+1}^n u_{ij} \otimes u_{ij}
\\
&=
\frac1{1-2\alpha} \left( 1 - \frac{\alpha^\ell}{(1-\alpha)^\ell} \right) {\rm id}_{n} - \frac{\alpha - \alpha^2}{1-2\alpha} \left( \frac1{1-\alpha} - \frac{\alpha^{\ell-1}}{(1-\alpha)^\ell} \right) m\, {\rm id}_{n}. 
\end{align*}
In view of $\alpha \in (0,\frac12)$, by taking the limit $\ell \to \infty$,  we derive that 
\begin{equation}\label{e:EigenVec}
\sum_{i=1}^m \sum_{j=k+1}^n u_{ij} \otimes u_{ij}
=
\frac{1-m\alpha}{1-2\alpha} {\rm id}_{n}. 
\end{equation}

Suppose that $\alpha \neq \frac1m$, and derive a contradiction. 
Then the right-hand side in \eqref{e:EigenVec} is isomorphic.
On the other hand, the left-hand side is the sum of $m(n-k)$ projections onto 1 dimensions. Thus, it follows that 
$
m (n-k) \ge n, 
$
which yields that $k \le (1- \frac1m)n$. This contradicts with $k > (1- \frac1m)n$. 
Therefore, we conclude that $\alpha=\frac1m$, and thus $k=n$ by \eqref{e:EigenVec}. 
\end{proof}

Let us conclude the proof of Theorem \ref{t:KW} by confirming the existence of the extremizer of \eqref{e:MinimizeGaussian-Tal}. 
As we mentioned, we will work on the entropic formulation, and use $\gamma_{A_i}$, rather than $g_{A_i}$, in below. 
Note that, according to Proposition 2.1 in \cite{KW}, we know that $ {\rm T}_{\mathbf{g}} >-\infty$. Also, the example $\mu_i = \gamma_{\rm id_n}$ tells us that $ {\rm T}_{\mathbf{g}}\le 0 $. 
The barycenter $\mu$ of $\gamma_{A_1},\ldots,\gamma_{A_m}$ is known to be Gaussian $\gamma_{A_0}$ where $A_0$ is uniquely determined by the nonlinear equation 
\begin{equation}\label{e:GaussianBrycenter}
A_0 = \frac1m \sum_{i=1}^m \big( A_0^\frac12 A_i A_0^\frac12 \big)^\frac12,  \end{equation}
see \cite{AC,ABCM}.
Entropy and Wasserstein distance of Gaussians are given by 
\begin{align*}
{\rm H}(\gamma_{A_i}|\gamma)
&= 
\frac12 {\rm Tr}\, A_i - \frac{n}2 - \frac12 \log\, {\rm det}\, A_i,\\
\frac1m \sum_{i=1}^m W_2^2(\gamma_{A_i},\gamma_{A_0}) 
&= 
\frac1m \sum_{i=1}^m \big(
{\rm Tr}\, A_0 + {\rm Tr}\, A_i - 2 {\rm Tr}\, ( A_0^\frac12 A_i A_0^\frac12 )^\frac12 
\big)\\
&= 
\frac1m \sum_{i=1}^m {\rm Tr}\, A_i - {\rm Tr}\, A_0,
\end{align*}
where we used \eqref{e:GaussianBrycenter}. 
Thus, 
\begin{align*}
\mathcal{T}(\gamma_{A_1},\ldots,\gamma_{A_m})
= 
\frac12 {\rm Tr}\, A_0 - \frac1{2m^2} \sum_{i=1}^m {\rm Tr}\, A_i - \frac{n(m-1)}{2m} - \frac{m-1}{2m^2} \sum_{i=1}^m \log\, {\rm det}\, A_i. 
\end{align*}
The iteration scheme to identify $A_0$ has been proposed by \cite{ABCM}, and they observed in Theorem 4.2 in their paper that 
\begin{equation}\label{e:ABCM-Inequ}
	{\rm Tr}\, S_k\le {\rm Tr}\, S_{k+1} \le {\rm Tr}\, A_0 \le \frac1m \sum_{i=1}^m {\rm Tr}\, A_i,
\end{equation}
where $S_0$ is a symmetric positive definite matrix that has been arbitrary chosen as an initial data, and 
$$
S_{k+1} := S_k^{-\frac12} \bigg( \frac1m \sum_{i=1}^m \big( S_k^\frac12 A_i S_k^\frac12 \big)^\frac12 \bigg)^2 S_k^{-\frac12}. 
$$
If we begin with the initial data $S_0 = {\rm id}_n$, then 
$$
S_1 = \bigg( \frac1m \sum_{i=1}^m A_i^\frac12 \bigg)^2
= 
\frac1{m^2} 
\sum_{i=1}^m
A_i 
+
\frac1{m^2} 
\sum_{i\neq j}
A_i^\frac12 A_j^\frac12, 
$$
and hence \eqref{e:ABCM-Inequ} particularly implies that $ {\rm Tr}\, A_0 $ is comparable to $ \sum_{i=1}^m {\rm Tr}\, A_i $. 
Moreover, 
$$
{\rm Tr}\, A_0 \ge \frac{1}{m^2} \sum_{i=1}^m {\rm Tr}\, A_i + \frac{1}{m^2} \sum_{i\neq j} {\rm Tr}\, \big(A_i^\frac14 A_j^\frac12 A_i^\frac14 \big),
$$
from which we obtain that 
\begin{align}\label{e:LowerBoundPhi}
\mathcal{T}(\gamma_{A_1},\ldots,\gamma_{A_m})
\ge \frac{1}{2m^2} 
\sum_{i\neq j} {\rm Tr}\, \big(A_i^\frac14 A_j^\frac12 A_i^\frac14 \big)
- \frac{m-1}{2m^2} \sum_{i=1}^m \log\, {\rm det}\, A_i - \frac{n(m-1)}{2m}. 
\end{align}

We claim from this inequality that if $\mathbf{A}$ is near extremizer of ${\rm T}_{\mathbf{g}}$ then none of eigenvalues of $A_1,\ldots,A_m$ does not diverge when $m\ge3$. 
Recalling that ${\rm T}_{\mathbf{g}}\in (-\infty,0]$, the near extremizer $\mathbf{A}$ must satisfy 
\begin{equation}\label{e:UniformBoundT}
    -\infty < {\rm T}_{\mathbf{g}} -1 \le \mathcal{T}(\gamma_{A_1},\ldots,\gamma_{A_m}) 
    \le 
    {\rm T}_{\mathbf{g}} + 1 <\infty.  
\end{equation} 
Firstly, we may assume that ${\rm det}\, A_i \sim1$ for all $i =1,\ldots,m$. 
This is because of \eqref{e:LowerBoundPhi} which in particular implies that 
$$
\mathcal{T}(\gamma_{A_1},\ldots,\gamma_{A_m})
\ge 
\frac{1}{2m^2} 
\sum_{i\neq j} 
\big( {\rm det}\, A_i {\rm det}\, A_j \big)^{\frac1{2n}}
- 
\frac{m-1}{2m^2} \sum_{i=1}^m \log\, {\rm det}\, A_i - \frac{n(m-1)}{2m}. 
$$
Hence, in view of \eqref{e:UniformBoundT}, $ {\rm det}\, A_i {\rm det}\, A_j  $ must be uniformly bounded from above. 
If ${\rm det}\, A_1 \to \infty$ for instance, then ${\rm det}\, A_j \le C / {\rm det}\, A_i \to 0$ for $j=2,\ldots,m$. However, in view of $m-1\ge2$, this means that $ \sum_{i=1}^m \log\, {\rm det}\, A_i \to \infty $ which is a contradiction. 

With this in mind, let $A_1, \dots, A_m$ be symmetric positive definite matrices satisfying that 
 $$
 {\rm Tr}\, (A_i A_j) \le C, \; \forall i \neq j, \;\;\; {\rm det}\, A_i \sim 1,\; \forall i \in [m]. 
  $$
 Without loss of generality, we may suppose that $A_1$ is diagonal and 
 $$
 \max_{k \in [n], i \in [m]} \lambda_{k}(A_i) = (A_1)_{11}. 
 $$
 Here $\lambda_1(A), \dots, \lambda_n(A)$ is the eigenvalues of $A$. 
 In what follows, let us denote for $i \in [m]$, 
 $$
 A_i =
 \begin{pmatrix}
     (A_i)_{11} & v_i \\
     v_i^* & \overline{A_i}, 
 \end{pmatrix}
 $$
 where $\overline{A_i} \in \mathbb{R}^{{(n-1)}\times {(n-1)}}$ and $v_i \in \mathbb{R}^{n-1}$. 

 First let us show that 
 \begin{equation}\label{e:Step1}
     {\rm Tr}\, \overline{A_i} \lesssim (A_1)_{11} \lesssim {\rm det}\, \overline{A_i}, \;\;\; \forall i=2, \dots, m. 
 \end{equation}
 Let us fix $i=2, \dots, m$. 
 To see above, since ${\rm Tr}\, (A_1A_i) \le C$ and $A_1$ is diagonal, it holds that 
 $$
 (A_1)_{11}(A_i)_{11} \le C, 
 $$
 which means that 
 \begin{equation}\label{e:Est1}
     (A_i)_{11} \lesssim \frac{1}{(A_1)_{11}}. 
 \end{equation}
 Moreover since ${\rm det}\, A_i \sim 1$, it follows from \eqref{e:Est1} that 
 $$
 1 \sim {\rm det}\, A_{i} \le (A_{i})_{11} {\rm det}\, \overline{A_{i}} \le \frac{1}{(A_1)_{11}} {\rm det}\, \overline{A_{i}},  
 $$
 which means 
 $$
 (A_1)_{11} \lesssim {\rm det}\, \overline{A_{i}}. 
 $$
 Next, since 
 $$
 (A_1)_{11}
 =
 \max_{k \in [n], i \in [m]} \lambda_{k}(A_i)
 \ge
 (A_i)_{kk}, \;\;\; \forall k=2, \dots, n, 
 $$
 we have 
 $$
 (A_1)_{11} \gtrsim \sum_{k=2}^n (A_i)_{kk} = {\rm Tr}\, \overline{A_{i}}. 
 $$

Secondly let us show that 
 \begin{equation}\label{e:Step2}
     |v_i| \lesssim 1, \;\;\;\forall i=2, \dots, m.
 \end{equation}
 To show this, for fixed $i=2, \dots, m$, we first note that  since $A_i>0$, it holds that 
 $$
 (A_i)_{11} > v_i \overline{A_i}^{-1} v_i^* \ge (\lambda_{{\rm max}}(\overline{A_i}))^{-1} |v_i|^2,  
 $$
 where $\lambda_{{\rm max}}(\overline{A_i}) \coloneqq \max_{k \in [n-1]} \lambda_k(\overline{A_i})$. 
 On the other hand, \eqref{e:Step1} implies that 
 $$
 \lambda_{\rm \max}(\overline{A_i}) \le {\rm Tr}\, \overline{A_i} 
 \lesssim (A_1)_{11}. 
 $$
 Hence combining \eqref{e:Est1}, it holds that 
 $$
 |v_i|^2 \le \lambda_{\rm \max}(\overline{A_i}) (A_i)_{11}
 \lesssim 1. 
 $$

 Thirdly note that 
 \begin{equation}\label{e:Step3}
 {\rm Tr}\, A_iA_j
 =
 (A_i)_{11}(A_j)_{11} + 2 \langle v_i, v_j \rangle + {\rm Tr}\, (\overline{A_i} \overline{A_j}), \;\;\; \forall i \neq j. 
 \end{equation}
 This identity is a conclusion of the direct calculation, so we omit the proof of it. 

 Finally let us recall that $A_i=A_i^{(\ell)}$, and suppose that $(A_1^{(\ell)})_{11} \to \infty$ as $\ell \to \infty$. 
 Let us fix $i, j\in \{2, \dots, m\}$ with $i \neq j$. 
 Then we apply \eqref{e:Step1} to see that 
 $$
 (A_1^{(\ell)})_{11}^{2} \lesssim {\det}\, (\overline{A_i^{(\ell)}} \overline{A_{j}^{(\ell)}} )
 \lesssim
 ({\rm Tr}\, (\overline{A_i^{(\ell)}} \overline{A_{j}^{(\ell)}} ) )^{\frac1{n-1}}. 
 $$
 Since $(A_1^{(\ell)})_{11} \to \infty$ as $\ell \to \infty$, this means that ${\rm Tr}\, (\overline{A_i^{(\ell)}} \overline{A_{j}^{(\ell)}}) \to \infty$ as $\ell \to \infty$. 
 On the other hand, \eqref{e:Step2} means that 
 $$
 \langle v_i, v_j \rangle
 \ge
 - |v_i| |v_j|
 \ge - C
 $$
 for some constant $C=C_{n,m}>0$. 
 Thus it follows from \eqref{e:Step3} that 
 $$
 {\rm Tr}\, (A_i^{(\ell)}A_j^{(\ell)}) 
 \ge
 - 2C + {\rm Tr}\, (\overline{A_i^{(\ell)}} \overline{A_{j}^{(\ell)}})
 \to \infty, \;\;\; \ell \to \infty,
 $$
 which is a contradiction. 

We now take the minimizing sequence $\mathbf{A}^{(R)}$: 
$$
{\rm T}_{\mathbf{g}} = \lim_{R\to \infty} \mathcal{T}( \gamma_{A_1^{(R)}},\ldots,\gamma_{A_m^{(R)}} ). 
$$
From the above claim,  none of eigenvalues of $A_i^{(R)}$ tends to infinity. 
Therefore, $A_i^{(R)}$ has some subsequence which converges to some $A_i^{\star}\ge0$ in a standard topology of $\mathbb{R}^{n^2}$. 
If $A_i^{\star}$ is degenerate in the sense that one of eigenvalue of $A_i^\star$ is zero for some $i$, then $\gamma_{A_i^\star}$ contains some Dirac delta. 
But, in such a case, $\mathcal{T}(\gamma_{A_1^\star},\ldots, \gamma_{A_m^{\star}}) = +\infty$ which is a contradiction. 
This concludes the existence of the minimizer $A_i^\star >0$.

\section{Appendix: Preservation of uniform log-concavity and log-convexity}\label{Appendix}

\begin{lemma}[{\cite[Theorem 4.3]{BraLi_JFA}}]\label{l:BraLi76}
	Let $\mathcal{Q} = \begin{pmatrix}
		A & B \\
		B^{*} & C  
	\end{pmatrix} $ be a positive definite matrix on $\mathbb{R}^{2n}$, where $A,B,C$ are $n\times n$ matrices. 
	Also, let\footnote{This $D$ comes from the fact that 
	$$
		\int_{\mathbb{R}^n} e^{-\langle (x,y), \mathcal{Q} (x,y)\rangle}\, dy = 
		c e^{-\langle x, Dx\rangle},\quad x\in \mathbb{R}^n,
	$$
	for some explicit constant $c$. } 
	$$
	D:= A - B C^{-1} B^*. 
	$$
	For a given log-concave $F:\mathbb{R}^{2n}\to [0,\infty)$, 
        let 
	$$
	G(x):= e^{\langle x, Dx\rangle} \int_{\mathbb{R}^n} e^{-\langle (x,y),\mathcal{Q}(x,y) \rangle} F(x,y)\, dy,\quad x\in \mathbb{R}^n. 
	$$
	Then, if $F$ is log-concave on $\mathbb{R}^{2n}$, $G$ is log-concave on $\mathbb{R}^n$. If $F$ is log-convex on $\mathbb{R}^{2n}$, $G$ is log-convex on $\mathbb{R}^n$
\end{lemma}

By using this, we may prove the following.

\begin{lemma}\label{l:ConvUniLog}
	Let $ 0< \lambda_1, \lambda_2 \le \Lambda_1, \Lambda_2 <\infty$ and $f_i \in L^1(\mathbb{R}^n)$ for $i=1,2$.  If $f_i$ is $\lambda_i$-uniformly log-concave, then $f_1 \ast f_2$ is $(\lambda_1^{-1} + \lambda_2^{-1})^{-1}$-uniformly log-concave. 
    Similarly, if $f_i$ is $\Lambda_i$-uniformly log-convex, then $f_1 \ast f_2$ is $(\Lambda_1^{-1} + \Lambda_2^{-1})^{-1}$-uniformly log-convex. 
	In particular, if $f \in L^1(\mathbb{R}^n)$ is $\lambda$-uniformly log-concave and $\Lambda$-uniformly log-convex for given $0<\lambda < \Lambda$, then so is $f \ast f (\sqrt{2} \cdot)$. 
\end{lemma}

\begin{proof}
Suppose that $f_i$ is $\lambda_i$-uniformly log-concave for $i=1,2$, and let us show that $f_1 \ast f_2$ is $(\lambda_1^{-1} + \lambda_2^{-1})^{-1}$-uniformly log-concave. 
The argument in the case of the uniformly log-convex is the same. 
Note that $g_i(x):=f_i(x)e^{\frac{\lambda_i}2|x|^2}$ for $i=1,2$ is log-concave, and our goal is to show 
\begin{equation}\label{e:Goal8/11-1}
	f_1\ast f_2(x) e^{ \frac{\lambda_1 \lambda_2}{2(\lambda_1 + \lambda_2)} |x|^2 }:\; \textrm{log-concave}. 
\end{equation}
To see this, we first notice from a direct calculation that 
\begin{align*}
f_1\ast f_2(x) e^{ \frac{\lambda_1 \lambda_2}{2(\lambda_1 + \lambda_2)} |x|^2 }
&= 
\int g_1(x-y)g_2(y) e^{ - \langle (x,y), \mathcal{Q}_0 (x,y) \rangle }\, dy, 
\end{align*}
where 
$
\mathcal{Q}_0:= 
\begin{pmatrix} (\frac{\lambda_1}2 - \frac{\lambda_1\lambda_2}{2(\lambda_1+\lambda_2)})\, {\rm id}_{n} & - \frac{\lambda_1}2\, {\rm id}_{n} \\ - \frac{\lambda_1}2\, {\rm id}_{n} & \frac{\lambda_1 + \lambda_2}{2} \, {\rm id}_{n} \end{pmatrix}.
$
Remark that this $\mathcal{Q}_0$ is positive semidefinite but not positive definite, and thus we consider 
$$
\mathcal{Q}= 
\mathcal{Q}_\varepsilon
:= 
\begin{pmatrix} (\frac{\lambda_1}2 - \frac{\lambda_1\lambda_2}{2(\lambda_1+\lambda_2)} + \varepsilon)\, {\rm id}_{n} & - \frac{\lambda_1}2\, {\rm id}_{n} \\ - \frac{\lambda_1}2\, {\rm id}_{n} & \frac{\lambda_1 + \lambda_2}{2} \, {\rm id}_{n} \end{pmatrix},
$$
for $\varepsilon>0$ instead. Then $\mathcal{Q}_\varepsilon$ is positive definite. That is, we consider 
$$
f_1\ast f_2(x) e^{ (\frac{\lambda_1 \lambda_2}{2(\lambda_1 + \lambda_2)} -\varepsilon) |x|^2 }
= 
\int g_1(x-y)g_2(y) e^{ - \langle (x,y), \mathcal{Q}_\varepsilon (x,y) \rangle }\, dy. 
$$
With this choice, $D = (\frac{\lambda_1}2 - \frac{\lambda_1 \lambda_2}{2(\lambda_1 + \lambda_2)} + \varepsilon - \frac{\lambda_1^2}{4} \frac{2}{\lambda_1+\lambda_2} ) {\rm id}_{n} = \varepsilon {\rm id}_{n}$, and Lemma \ref{l:BraLi76} confirms that 
$$
G(x):= e^{\varepsilon|x|^2} 
\int g_1(x-y)g_2(y) e^{ - \langle (x,y), \mathcal{Q}_\varepsilon (x,y) \rangle }\, dy = f_1\ast f_2(x) e^{ \frac{\lambda_1 \lambda_2}{2(\lambda_1 + \lambda_2)} |x|^2 }
$$
is log-concave if $F(x,y):= g_1(x-y)g_2(y)$ is log-concave on $\mathbb{R}^{2n}$. 
The log-concavity of $F$ is a consequence of the one of $g_i$ for $i=1,2$. 
Thus, we complete the proof of \eqref{e:Goal8/11-1}. 
\end{proof}

\section*{Acknowledgements}
This work was supported by JSPS Overseas Research Fellowship and JSPS Kakenhi grant numbers 21K13806, 23K03156, and 23H01080 (Nakamura), and JSPS Kakenhi grant numbers 24KJ0030 (Tsuji). 
The first author is grateful to Naohito Tomita and Takahisa Inui for their generous supports while he was engaging to this work. He also thanks to analysis group in University of Birmingham for their hospitality. 
The second author would like to thank Neal Bez for introducing Ball's inequality and telling him related topics on it. 
The second author also would like to thank Takuya Nishisako for discussing and giving him useful comments.

\end{document}